\documentclass{amsart}

\usepackage{amssymb}
\usepackage{appendix}
\usepackage{multicol}
\usepackage{enumitem}
\usepackage{xcolor}
\usepackage[bottom=2.5cm,top=2.5cm,left=2.5cm,right=2.5cm]{geometry}
\usepackage{hyperref}
\hypersetup{ 
colorlinks=true,       		
linkcolor=blue,
citecolor=blue,
urlcolor=blue,
}

\newcommand{\df}[1]{\textbf{#1}}
\newcommand{\dha}{d_H}

\numberwithin{equation}{section}

\newtheorem{theorem}{Theorem}[section]
\newtheorem{lemma}[theorem]{Lemma}
\newtheorem{proposition}[theorem]{Proposition}
\newtheorem{corollary}[theorem]{Corollary}

\theoremstyle{definition}
\newtheorem{definition}[theorem]{Definition}

\newtheorem{remark}[theorem]{Remark}

\newcommand{\describeContent}[1]{%
\begingroup%
\begin{NoHyper}
\let\thefootnote\relax%
\footnotetext{#1}%
\end{NoHyper}
\endgroup%
}

\allowdisplaybreaks

\title{Generalized exponential pullback attractor for a nonautonomous wave equation}

\author{Matheus C. Bortolan}
\author{Tom\'as Caraballo}
\author{Carlos Pecorari Neto}

\begin{document}

\describeContent{
[M.C. Bortolan] This study was financed in part by the Coordena\c c\~ao de Aperfei\c coamento de Pessoal de N\'ivel Superior - Brasil (CAPES) - Finance Code 001 and by the Ministerio de Ciencia e Innovaci\'on, Agencia Estatal de Investigaci\'on (AEI) and FEDER under project PID2021-122991NB-C21.
\par \quad Address: Departamento de Matem\'atica, Centro de Ci\^encias F\'isicas e Matem\'aticas, Universidade Federal de Santa Catarina (UFSC), Campus Florian\'opolis, CEP 88040-090, Florian\'opolis - SC, Brasil.

e-mail: \texttt{m.bortolan@ufsc.br}, ORCiD: 0000-0002-4838-8847.

[T. Caraballo] This study was financed in part by Ministerio de Ciencia e Innovaci{\'o}n (MCI), Agencia Estatal de Investigaci{\'o}n (AEI) and Fondo Europeo de Desarrollo Regional (FEDER) under the project PID2021-122991NB-C21.
\par \quad Address: Departamento de Ecuaciones Diferenciales y An\'alisis Num\'erico, Universidad de Sevilla, Campus Reina Mercedes, 41012, Seville, Spain. e-mail: \texttt{caraball@us.es}, ORCiD: 0000-0003-4697-898X.

[C. Pecorari Neto] This study was financed in part by the Coordena\c c\~ao de Aperfei\c coamento de Pessoal de N\'ivel Superior - Brasil (CAPES) - Finance Code 001.
\par \quad Address: Departamento de Matem\'atica, Centro de Ci\^encias F\'isicas e Matem\'aticas, Universidade Fe\-de\-ral de Santa Catarina (UFSC), Campus Florian\'opolis, CEP 88040-090, Florian\'opolis - SC, Brasil. 

e-mail: \texttt{carlospecorarineto@gmail.com}, ORCiD: 0009-0007-4218-7781.
}

\maketitle

\begin{abstract}
In this work we introduce the concept of \textit{generalized exponential $\mathfrak{D}$-pullback attractor} for evolution processes, where $\mathfrak{D}$ is a universe of families in $X$, which is a compact and positively invariant family that pullback attracts all elements of $\mathfrak{D}$ with an \textit{exponential rate}. Such concept was introduced in \cite{Pecorari} for the general case of decaying functions (which include the exponential decay), but for fixed bounded sets rather than to universe of families. We prove a result that ensures the existence of a generalized exponential $\mathfrak{D}_{\mathcal{C}^\ast}$-pullback attractor for an evolution process, where $\mathfrak{D}_{\mathcal{C}^\ast}$ is a specific universe. This required an adaptation of the results of \cite{Pecorari}, which only covered the case of a polynomial rate of attraction, for fixed bounded sets. Later, we prove that a nonautonomous wave equation has a generalized exponential $\mathfrak{D}_{\mathcal{C}^\ast}$-pullback attractor. This, in turn, also implies the existence of the $\mathfrak{D}_{\mathcal{C}^\ast}$-pullback attractor for such problem.

 Non\-au\-tonomous wave equation. Evolution processes. Pullback attractors.

\end{abstract}

\section{Introduction}

The \textit{pullback attractor} for an \textit{evolution process}, in general terms, represents the sets of limit states of the solutions of a nonautonomous equation in the pullback sense, and also contain all the bounded solutions defined for all time (see \cite{Bortolan2020,Caraballo2006,CRLBook}, for instance). In other words, it contains all the \textit{physically relevant solutions}, bearing in mind real world problems from Physics, Chemistry, Biology, Economy, etc. 

The problem is that there might be no information regarding the \textit{rate of attraction} of an existing pullback attractor. Many authors have worked with the notion of a  \textit{pullback exponential attractor} (see for instance \cite{Sonner1,Sonner2}) for an evolution process, that is, a compact and positively invariant family that \textit{exponentially pullback attracts} bounded subsets, which moreover, have a (uniformly) finite fractal dimension. We slightly generalize this notion, adapting the ideas of \cite{Pecorari}, dropping the requirement of finitude of the fractal dimension and adding the property of \textit{exponential pullback attraction of universes}, which we call a \textit{generalized exponential $\mathfrak{D}$-pullback attractor}, where $\mathfrak{D}$ is a given \textit{universe} (which we define later). Our main result, also inspired by \cite{ZhaoZhong2022}, and following the ideas of \cite{Pecorari}, is to prove the existence of such an object, for a suitable universe $\mathfrak{D}_{\mathcal{C}^\ast}$, for the nonautonomous wave equation given by:
\begin{equation}\label{ourproblem}\tag{NWE}
\left\{\begin{aligned}
& u_{tt}(\!t,\!x)\!-\!\Delta u(t,x)\!+\!k(t)u_t(t,x)\!+\!f(t,u(t,x))\!=\! \int_{\Omega}K(x,y)u_t(t,y)dy\!+\!h(x), (t,x)\!\in\! [s,\infty)\!\times\!\Omega,\\
& u(t,x) = 0, (t,x)\in \left[s,\infty \right) \times \partial\Omega,\\
& u(s,x)=u_0(x), \ u_t(s,x)=u_1(x), x\in \Omega,
\end{aligned} \right.
\end{equation}
where $\Omega\subset \mathbb{R}^3$ is a bounded domain with smooth boundary $\partial\Omega$, and we have the following conditions:
\begin{enumerate}[leftmargin=*,label={\rm \bfseries (H$_{\arabic*}$)}]
\item \label{cond1} $h\in L^2(\Omega)$ and we set $h_0:=\|h\|_{L^2(\Omega)}$,
\item \label{cond2} $K\in L^2(\Omega \times \Omega)$ and we set $K_0:=\|K\|_{L^2(\Omega\times \Omega)}$,
\item \label{cond3} $k\colon \mathbb{R}\rightarrow (0,\infty)$ is a continuous function satisfying $0\leqslant K_0<k_0\leqslant k(t)\leqslant k_1$ for all $t\in\mathbb{R}$, where $k_0,k_1$ are constants and $K_0$ is the constant of \ref{cond2};
\item \label{cond4} $f\in C^1(\mathbb{R}^2, \mathbb{R})$ satisfies
\[
\liminf_{|v|\to \infty}\Big(\inf_{t\in\mathbb{R}}\frac{\partial f}{\partial v}(t,v)\Big) > -\lambda_1,
\quad \hbox{ and } \quad \liminf_{|v|\to \infty}\Big(\inf_{t\in\mathbb{R}}\frac{f(t,v)}{v}\Big) > -\lambda_1,
\]
and there exists a continuous function $c_0\colon \mathbb{R} \rightarrow (0,\infty)$ such that for all $t,v\in \mathbb{R}$ we have
\[
|f(t,0)| \leqslant c_0(t), \quad  \left|\frac{\partial f}{\partial v}(t,v)\right| \leqslant c_0(t) (1+|v|^2),\\
\quad \hbox{ and } \quad \left|\frac{\partial f}{\partial t}(t,v)\right| \leqslant c_0(t),
\]
where 
$\lambda_1>0$ is the first eigenvalue of the negative Laplacian operator $-\Delta$ with Dirichlet boundary conditions in $\Omega$, that is, of the operator $A:=-\Delta \colon H^1_0(\Omega)\cap H^2(\Omega)\subset L^2(\Omega)\to L^2(\Omega)$, which is positive and selfadjoint, with compact resolvent.
\item \label{cond5} For $t,v\in \mathbb{R}$ we define the function 
\begin{equation*}\label{def.F}
F(t,v)=\displaystyle\int_{0}^v f(t,\xi)d\xi,
\end{equation*}
for which we assume for all $t\in \mathbb{R}$ we have
\[
\int_{\mathbb{R}}\left|\frac{\partial F}{\partial t}(t,v)\right|dv \leqslant c_0(t).
\]
\end{enumerate}

To be more precise about our goals in this paper, we begin by presenting an overview of the theory of evolution processes and their pullback attractors in metric spaces. In what follows we write $\mathbb{R}$ to denote the set of the real numbers. By setting $\mathcal{P}=\{(t,s)\in \mathbb{R}^2\colon t\geqslant s\}$ and considering $(X,d)$ a complete metric space, we say that a two-parameter family $S=\{S(t,s)\colon (t,s)\in \mathcal{P}\}$ of continuous maps from $X$ into itself is an \df{evolution process}\index{evolution process} if
\begin{enumerate}[label=$\circ$]
\item $S(t,t)x=x$ for all $x\in X$ and $t\in \mathbb{R}$;
\item $S(t,r)S(r,s)=S(t,s)$ for all $(t,r), (r,s)\in \mathcal{P}$, that is, $t,r,s\in \mathbb{R}$ and $t\geqslant r \geqslant s$;
\item the map $\mathcal{P}\times X \ni (t,s,x) \mapsto S(t,s)x\in X$ is continuous.
\end{enumerate}

We let $\mathfrak{F}$\index{class of families}\index{family} be the class of all families $\hat{D}=\left\{D_t\right\}_{t\in\mathbb{R}}$, where $D_t$ is a nonempty subset of $X$ for each $t\in\mathbb{R}$. A family $\hat{A}\in \mathfrak{F}$ is said to be \df{closed/compact}\index{family!closed}\index{family!compact} if $A_t$ is a closed/compact subset of $X$ for each $t\in \mathbb{R}$. Moreover, given $\hat{A},\hat{B}\in\mathfrak{F}$ we say that $\hat{A}\subset \hat{B}$ if $A_t\subset B_t$ for each $t\in\mathbb{R}$. A subclass $\mathfrak{D}\subset \mathfrak{F}$ is called a \df{universe} if it is \textit{closed by inclusion}, that is, given $\hat{D}_1\in\mathfrak{D}$ and $\hat{D}_2\in \mathfrak{F}$ with $\hat{D}_2\subset \hat{D}_1$ we have $\hat{D}_2\in \mathfrak{D}$.

Recall that for an evolution process $S$ in $X$ and a universe $\mathfrak{D}$, we say that a family $\hat{B}\in \mathfrak{F}$ is \df{$\mathfrak{D}$-pullback attracting} if for each $t\in \mathbb{R}$ and $\hat{D}\in\mathfrak{D}$ we have
\[
\lim_{s\to  -\infty}\dha (S(t,s)D_s,B_t)=0,
\]
where
\[
\dha (U,V)=\sup\limits_{u\in U}\inf\limits_{v\in V}d(u,v)
\]
denotes the \textit{Hausdorff semidistance}\index{Hausdorff semidistance} between two nonempty subsets $U$ and $V$ of $X$. 

Consider a given evolution process $S$ in $X$. We say that a family $\hat{B}$ is \df{invariant/positively invariant}\index{family!invariant}\index{family!positively invariant} for $S$ if for all $(t,s)\in \mathcal{P}$ we have $S(t,s)B_s = B_t$ / $S(t,s)B_s\subset B_t$, respectively.  For an evolution process $S$ in a metric space $X$, we say that $\hat{A}\in\mathfrak{F}$ is a \df{$\mathfrak{D}$-pullback attractor} for $S$ if:
\begin{enumerate}[label=(\roman*)]
\item $\hat{A}$ is compact;
\item $\hat{A}$ is invariant for $S$;
\item \label{ppa3} $\hat{A}$ is $\mathfrak{D}$-pullback attracting;
\item \label{ppa4} $\hat{A}$ is the minimal closed family satisfying \ref{ppa3}, that is, if $\hat{C}\in\mathfrak{F}$ is a closed $\mathfrak{D}$-pullback attracting family, then $\hat{A}\subset\hat{C}$.
\end{enumerate}

The minimality condition \ref{ppa4} ensures that when a pullback attractor for $S$ exists, it is unique.  With that, given an evolution process $S$ in a metric space $X$, we can define the main object of this work.

\begin{definition}\label{def:GenPA}
We say that a family $\hat{M}\in \mathfrak{F}$ is a \df{generalized exponential $\mathfrak{D}$-pullback attractor} for $S$ if $\hat{M}$ is compact, positively invariant and it is \df{exponentially $\mathfrak{D}$-pullback attracting}, that is, there exists a constant $\omega>0$ such that for \textit{every} $\hat{D}\in \mathfrak{D}$ and $t\in\mathbb{R}$ there exist $C=C(\hat{D},t)\geqslant 0$ and $\tau_0=\tau_0(\hat{D},t)\geqslant 0$ such that 
\[
\dha (S(t,t-\tau)D_{t-\tau},M_t)\leqslant C e^{-\omega \tau} \quad \hbox{ for all } \tau\geqslant \tau_0.
\]
\end{definition}

We consider\footnote[1]{Note that if $r_1,r_2\in {\mathcal{C}^\ast}$ and $c>0$ is a constant, then $r_1+r_2\in {\mathcal{C}^\ast}$, $cr_1\in {\mathcal{C}^\ast}$, $r_1r_2\in {\mathcal{C}^\ast}$ and $\sqrt{r_1}\in {\mathcal{C}^\ast}$.}
\[
{\mathcal{C}^\ast} = \left\{r\colon \mathbb{R}\to (0,\infty) \ \Big| \  r \hbox{ is continuous and } \lim_{\tau\to \infty}\sup_{s\leqslant t}r(s-\tau)e^{-\alpha \tau}=0 \hbox{ for each } \alpha>0 \hbox{ and } t\in \mathbb{R}\right\},
\]
and, for $(X,\|\cdot\|)$ a Banach space, we define the universe
\begin{equation}\label{eq.def.universe}
\mathfrak{D}_{\mathcal{C}^\ast} = \left\{\hat{D}\colon \hbox{ there exists }r \in {\mathcal{C}^\ast} \hbox{ with } D_t \subset \overline{B}^X_{r(t)} \hbox{ for all } t\in \mathbb{R}\right\},
\end{equation}
where $\overline{B}^X_r=\{x\in X\colon \|x\|\leqslant r\}$. We assume also that the function $c_0$ that appears in both \ref{cond4} and \ref{cond5} is such that
\begin{enumerate}[leftmargin=*,label={\rm \bfseries (H$_6$)}]
    \item \label{cond6} $c_0\in {\mathcal{C}^\ast}$.
\end{enumerate}

The main result of this paper is the following:

\begin{theorem}\label{App:Att}
Assume that \ref{cond1}-\ref{cond6} hold true. Then the evolution process $S$ associated with \eqref{ourproblem} in $X:=H^1_0(\Omega)\times L^2(\Omega)$ possesses a generalized exponential $\mathfrak{D}_{\mathcal{C}^\ast}$-pullback attractor $\hat{M}\in \mathfrak{D}_{\mathcal{C}^\ast}$ in $X$. Furthermore, $S$ has a $\mathfrak{D}_{\mathcal{C}^\ast}$-pullback attractor $\hat{A}$, with $\hat{A}\subset \hat{M}$.
\end{theorem}

\section{Generalized exponential \texorpdfstring{$\mathfrak{D}$}{D}-pullback attractors} \label{Dphitheory}

Now, we extend the theory presented in \cite{Pecorari} for the case of \textit{universes}, that is, we replace the attraction of fixed bounded sets with the attraction of families in a given universe. There is no substantial change if we consider a general \textit{decay function} $\varphi$ or the particular case of an exponentially decaying function, therefore, since our application deals with the exponential case, we present the theory only for the exponential function. 

For what follows, unless clearly stated otherwise, $(X,d)$ denotes a \textit{complete} metric space. Recall that for a nonempty bounded subset $C\subset X$, its \textit{diameter} is defined as $\operatorname{diam}(C) := \sup_{x,y\in C}d(x,y)$. For $x_0\in X$ and $r>0$, the \textit{closed ball of radius $r$ centered in $x_0$} will be denoted by
$\overline{B}_r(x_0):= \{x\in X\colon d(x,x_0)\leqslant r\}$. 
When there is a need to highlight the space $X$ in which the balls are being considered, we will use the notation $\overline{B}_r^X(x_0)$. For $r>0$ and a nonempty set $A$, we denote $\mathcal{O}_r(A)=\{x\in X\colon d(x,a)<r \hbox{ for some } a\in A\}$. For a bounded set $B\subset X$ we recall that its \textit{Kuratowski measure of non-compactness} is defined by
\begin{align*}
\kappa(B)=\inf\{\delta > 0\colon B & \hbox{ admits a finite cover by sets of diameter less than or equal to }\delta\}.
\end{align*}
Additionally, we define the \df{ball measure} \df{of non-com\-pact\-ness}\index{measure!ball} by
\[
\beta(B)=\inf\left\{r > 0\colon B \hbox{ admits a finite cover by open balls of radius } r\right\}.
\]
Note that for each bounded set $B\subset X$ it holds that $\beta(B)\leqslant \kappa(B)\leqslant 2\beta(B)$.

\begin{definition}\label{abcde} \label{def:UnifPA}
We say that an evolution process $S$ in $X$ is \df{exponentially $\mathfrak{D}$-pullback $\kappa$-dissipative} if there exists $\omega>0$ such that for every $\hat{D}\in\mathfrak{D}$ and $t\in\mathbb{R}$ there exists $C=C(t,\hat{D})\geqslant 0$ and $\tau_0=\tau_0(t,\hat{D})\geqslant 0$ such that
\[
\kappa\Big(\bigcup_{\sigma \geqslant \tau}S(t,t-\sigma)D_{t-\sigma}\Big)\leqslant Ce^{-\omega\tau} \quad \hbox{ for all } \tau \geqslant \tau_0.
\]

Also, we say that $\hat{B}\in\mathfrak{F}$ is \textbf{uniformly $\mathfrak{D}$-pullback absorbing} if given $\hat{D}\in\mathfrak{D}$ and $t\in\mathbb{R}$, there exists $T=T(t,\hat{D})>0$ such that $S(s,s-r)D_{s-r}\subset B_s$ for all $s\leqslant t$ and $r\geqslant T$.
\end{definition}

With minor changes to the proofs, following \cite{Pecorari}, we can prove the following results:

\begin{proposition}\label{continuouscase}
Let $S$ be an exponentially $\mathfrak{D}$-pullback $\kappa$-dissipative evolution process on $X$ and assume that there exists a closed family $\hat{B}=\left\{B_t\right\}_{t\in\mathbb{R}}$ of bounded sets, which is uniformly $\mathfrak{D}$-pullback absorbing and positively invariant. Suppose also that for each $s\in \mathbb{R}$ there exist $\gamma>0$ and $L>0$ such that for all $0\leqslant \tau \leqslant \gamma$ 
\[
d(S(s+\tau,s)x,S(s+\tau,s)y)\leqslant Ld(x,y) \quad \hbox{ for all } x,y\in B_s.
\]
Then there exists a generalized exponential $\mathfrak{D}$-pullback attractor $\hat{M}$ for $S$, with $\hat{M}\subset \hat{B}$.
\end{proposition}
 
\begin{proposition}\label{theo:GenimpliesPA}
If $S$ is an exponentially $\mathfrak{D}$-pullback $\kappa$-dissipative evolution process in $X$ with a generalized exponential $\mathfrak{D}$-pullback attractor $\hat{M}\in \mathfrak{D}$ and there exists $r>0$ such that $\hat{M}_r = \{\mathcal{O}_r(M_t)\}_{t\in \mathbb{R}}\in \mathfrak{D}$, then $S$ has a $\mathfrak{D}$-pullback attractor $\hat{A}$, with $\hat{A}\subset \hat{M}$.
\end{proposition}

We note that if $\hat{D}\in \mathfrak{D}_{\mathcal{C}^\ast}$ then $\hat{D}_r = \{\mathcal{O}_r(D_t)\}_{t\in \mathbb{R}}\in \mathfrak{D}_{\mathcal{C}^\ast}$ for each $r>0$.

\subsection{Existence of generalized exponential \texorpdfstring{$\mathfrak{D}_{\mathcal{C}^\ast}$}{DCast}-pullback attractors}

We first present a proposition that will be paramount for the proof of the existence of a generalized exponential $\mathfrak{D}_{\mathcal{C}^\ast}$-pullback attractor for an evolution process $S$.

\begin{proposition}\label{lemmaxnz}
Let  $\hat{B}\in\mathfrak{F}$ be a uniformly $\mathfrak{D}$-pullback absorbing family, where $\mathfrak{D}$ is any given universe. Suppose that there exists $\omega>0$ such that for each $t\in\mathbb{R}$ there exist $C\geqslant 0$, $\tau_0>0$ such that 
\[
\kappa(S(t,t-\tau)B_{t-\tau})\leqslant Ce^{-\omega\tau} \quad \hbox{ for all } \tau\geqslant \tau_0.
\]
Then $S$ is exponentially $\mathfrak{D}$-pullback $\kappa$-dissipative.
\begin{proof}
Let $\hat{D}\in\mathfrak{D}$ and $t\in\mathbb{R}$. Since $\hat{B}$ is uniformly $\mathfrak{D}$-pullback absorbing, there exists $T>0$ such that $S(s,s-r)D_{s-r}\subset B_s$ for all $s\leqslant t$ and $r\geqslant T$. Take $\sigma \geqslant \tau_0+T>0$ and note that if $s\geqslant 2\sigma$, since $t-\sigma\leqslant t$ and $s-\sigma\geqslant T$, we have
\begin{align*}
S(t,t-s)D_{t-s}&=S(t,t-\sigma)S(t-\sigma,t-s)D_{t-s}\\
&=S(t,t-\sigma)S(t-\sigma,(t-\sigma)-(s-\sigma))D_{(t-\sigma)-(s-\sigma)}\subset S(t,t-\sigma)B_{t-\sigma},
\end{align*}
which implies that $\bigcup_{s\geqslant 2\sigma}S(t,t-s)D_{t-s} \subset S(t,t-\sigma)B_{t-\sigma}$. Thus, since $\sigma\geqslant \tau_0$, it follows that
\begin{align*}
\kappa\Big(\bigcup_{s\geqslant 2\sigma}S(t,t-s)D_{t-s}\Big)\leqslant \kappa(S(t,t-\sigma)B_{t-\sigma})\leqslant Ce^{-\omega\sigma},
\end{align*}
for all $\sigma \geqslant T+\tau_0$. This is equivalent to 
\begin{align*}
\kappa\Big(\bigcup_{s\geqslant \tau}S(t,t-s)D_{t-s}\Big)\leqslant Ce^{-\tfrac{\omega}{2} \tau},
\end{align*}
for all $\tau \geqslant 2\tau_0+2T$, and the proof is complete.
\end{proof}
\end{proposition}

A function $\psi\colon X\times X \to \mathbb{R}^+$ is called \textit{contractive} on $B\subset X$ if for each sequence $\left\{x_n\right\}_{n\in\mathbb{N}}\subset B$ we have $\liminf_{m,n\to \infty}\psi(x_n,x_m)=0$. We denote the set of such functions by $\operatorname{\operatorname{contr}}(B)$. It is simple to see that a function $\psi$ is contractive on $B$ if for each sequence $\{x_n\}\subset B$ there exists a subsequence $\{x_{n_k}\}$ such that $\lim_{k,\ell\to \infty}\psi(x_{n_k},x_{n_\ell})=0$. A \textit{pseudometric} in $X$ is a function $\rho\colon X\times X\to [0,\infty)$ that satisfies:
\begin{enumerate}[label={$\circ$}]
    \item $\rho(x,x)=0$ for all $x\in X$;
    \item $\rho(x,y)=\rho(y,x)$ for all $x,y\in X$;
    \item $\rho(x,z) \leqslant \rho(x,y) + \rho(y,z)$ for all $x,y,z\in X$.
\end{enumerate}
We say that a pseudometric $\rho$ is \textit{precompact on $B\subset X$} if given $\delta>0$, there exists a finite set of points $\left\{x_1,...,x_r\right\}\subset B$ such that $B\subset \bigcup_{j=1}^r B_{\delta}^\rho(x_j)$, where $B_{\delta}^{\rho}(x_j)=\left\{y\in X\colon \rho(y, x_j)<\delta\right\}.$
We know that $\rho$ is precompact on $B$ if and only if any sequence $\left\{x_n\right\}\subset B$ has a Cauchy subsequence $\left\{x_{n_j}\right\}$ with respect to $\rho$.

In what follows, we prove the existence of a generalized exponential $\mathfrak{D}$-pullback attractor in the specific case of the universe $\mathfrak{D}_{\mathcal{C}^\ast}$ defined in \eqref{eq.def.universe}. 

\begin{theorem}\label{corMain2}
Let $X$ be a complete metric space and $S$ be a continuous evolution process in $X$ such that there exists a closed, positively invariant and uniformly $\mathfrak{D}_{\mathcal{C}^\ast}$-pullback absorbing family $\hat{B}\in \mathfrak{D}_{\mathcal{C}^\ast}$ for $S$. Suppose that there exists $\gamma>0$ such that for each $s\in\mathbb{R}$ and $0\leqslant \tau \leqslant \gamma$ there exists a constant $L_{\tau,s}>0$ such that
\begin{equation*}
    d(S(s+\tau,s)x,S(s+\tau,s)y)\leqslant L_{\tau,s}d(x,y) \hbox{ for all }x,y\in B_s.
\end{equation*}
Assume also that there exist $\mu\in (0,1)$, $T>0$, $r>0$ satisfying: given $t\in\mathbb{R}$, there exist functions $g_n\colon (\mathbb{R}^+)^m\to \mathbb{R}^+$ and $\psi_n\colon X \times X \to \mathbb{R}^+$ for each $n\in \mathbb{N}$, and pseudometrics $\rho_1,...,\rho_m$ on $X$ such that
\begin{enumerate}[label={(\roman*)}]
\item \label{itemi} $g_n$ is non-decreasing with respect to each variable, $g_n(0,...,0)=0$ and $g_n$ is continuous at $(0,...,0)$ for each $n\in \mathbb{N}$;
 \item for each $n\in\mathbb{N}$, the pseudometrics $\rho_1,..., \rho_m$ are precompact on $B_{t-nT}$;
\item $\psi_n\in \operatorname{contr}(B_{t-nT})$ for each $n\in\mathbb{N}$;
\item \label{itemiii} for each $n\in\mathbb{N}$ and $x,y\in B_{t-nT}$ we have
\begin{align*}
d(S_nx, S_ny)^r \leqslant \mu d(x,y)^r+g_n(\rho_1(x,y),..., \rho_m(x,y))+\psi_n(x,y),
\end{align*}
\end{enumerate}
where $S_n:= S(t-(n-1)T,t-nT)$ for each $n\in\mathbb{N}$.
 
Then $S$ is  exponentially $\mathfrak{D}_{\mathcal{C}^\ast}$-pullback $\kappa$-dissipative and it possesses a generalized exponential $\mathfrak{D}_{\mathcal{C}^\ast}$-pullback attractor $\hat{M}\in \mathfrak{D}_{\mathcal{C}^\ast}$. 
\begin{proof}
For $A\subset B_{t-T}$ and $\varepsilon>0$, there exist sets $E_1,...,E_p$ such that
\begin{align*}
A\subset \bigcup_{j=1}^p E_j \quad \hbox{ and } \quad \operatorname{diam}(E_j)< \kappa(A)+\varepsilon \hbox{ for } j=1,..,p. 
\end{align*}

If $\left\{x_i\right\}\subset A$, there exist $j\in\left\{1, \cdots, p\right\}$ and a subsequence $\left\{x_{i_k}\right\}$ of $\left\{x_i\right\}$ such that $\left\{x_{i_k}\right\}\subset E_j$. Thus,
\begin{equation}\label{subseq1}
    d(x_{i_k},x_{i_l})\leqslant \operatorname{diam}(E_j)<\kappa(A)+\varepsilon \hbox{ for all }k,l\in\mathbb{N}.
\end{equation}

Since $\rho_1, \cdots, \rho_m$ are precompact on $B_{t-T}$ and $\psi_1\in \operatorname{contr}(B_{t-T})$, we have
\begin{equation}  \label{subseq2} \liminf_{k,l\rightarrow\infty}g_1(\rho_1(x_{i_k},x_{i_l}),\cdots, \rho_m(x_{i_k},x_{i_l}))=0 \hbox{ and } \liminf_{k,l\rightarrow \infty}\psi_1(x_{i_k},x_{i_l})=0.
\end{equation}

Joining \eqref{subseq1}, \eqref{subseq2} and hypothesis \ref{itemiii}, we obtain
\begin{align*}
& \liminf_{k,l\rightarrow \infty}d(S_1x_{i_k},S_1x_{i_l})^r=\liminf_{k,l\rightarrow \infty}d(S(t,t-T)x_{i_k},S(t,t-T)x_{i_l})^r\\
&\quad \leqslant \liminf_{k,l\rightarrow \infty} [\mu d(x_{i_k},x_{i_l})^r+g_1(\rho_1(x_{i_k},x_{i_l}),\cdots, \rho_m(x_{i_k},x_{i_l}))+\psi_1(x_{i_k},x_{i_l})]\leqslant \mu(\kappa(A)+\varepsilon)^r,
\end{align*}
and since $\varepsilon>0$ is arbitrary, we conclude that for any sequence $\{x_n\}_{n\in \mathbb{N}}\subset A$ we have
\begin{equation*}
\liminf_{k,\ell\to \infty}d(S_1x_k,S_1x_\ell)^r\leqslant \mu\kappa(A)^r.
\end{equation*}

Now, let $A\subset B_{t-2T}$, $\varepsilon>0$ and $\left\{x_i\right\}\subset A$. As before, there exists a subsequence $\left\{x_{i_k}\right\}$ for which $d(x_{i_k},x_{i_l})<\kappa(A)+\varepsilon$ for all $k,l\in\mathbb{N}$. Since $\rho_1, \cdots, \rho_m$ are precompact on $B_{t-T}$ and $B_{t-2T}$, $\psi_1\in \operatorname{contr}(B_{t-T})$ and $\psi_2\in  \operatorname{contr}(B_{t-2T})$, and $S_2B_{t-2T}\subset B_{t-T}$, we obtain
\begin{equation*}    \liminf_{k,l\rightarrow\infty}g_2(\rho_1(x_{i_k},x_{i_l}),\cdots, \rho_m(x_{i_k},x_{i_l}))=0 \hbox{ , } \liminf_{k,l\rightarrow \infty}\psi_2(x_{i_k},x_{i_l})=0,
\end{equation*}
\begin{equation*}
\liminf_{k,l\rightarrow\infty}g_1(\rho_1(S_2x_{i_k},S_2x_{i_l}),\cdots, \rho_m(S_2x_{i_k},S_2x_{i_l}))=0 \hbox{ , } \liminf_{k,l\rightarrow \infty}\psi_1(S_2x_{i_k},S_2x_{i_l})=0.
\end{equation*}

Since for any $x,y \in B_{t-2T}$,
\begin{align*}
    &d(S(t,t-2T)x,S(t,t-2T)y)^r=d(S_1S_2x,S_1S_2y)^r\\
    &\quad \leqslant \mu d(S_2x,S_2y)^r+g_1(\rho_1(S_2x,S_2y),\ldots, \rho_m(S_2x,S_2y))+\psi_1(S_2x,S_2y)\\
    &\quad \leqslant \mu[\mu d(x,y)^r+g_2(\rho_1(x,y),\ldots, \rho_m(x,y))+\psi_2(x,y)]\\
    &\qquad+g_1(\rho_1(S_2x,S_2y),\ldots, \rho_m(S_2x,S_2y))+\psi_1(S_2x,S_2y),
\end{align*}
we obtain
\begin{align*}
    \liminf_{k,l\rightarrow\infty}d(S(t,t-2T)x_{i_k},S(t,t-2T)x_{i_l})^r\leqslant \mu^2(\kappa(A)+\varepsilon)^r,
\end{align*}
and, again, since $\varepsilon>0$ is arbitrary, we conclude that for any sequence $\{x_n\}_{n\in \mathbb{N}}\subset A$,
\begin{equation*}
  \liminf_{m,p\rightarrow\infty}d(S(t,t-2T)x_m,S(t,t-2T)x_p)^r\leqslant \mu^2\kappa(A)^r.
\end{equation*}

Inductively, for any $n\in\mathbb{N}$, $A\subset B_{t-nT}$ and $\{x_n\}_{n\in \mathbb{N}}\subset A$ we obtain
\begin{equation}\label{xcva}
    \liminf_{m,p\rightarrow\infty}d(S(t,t-nT)x_m,S(t,t-nT)x_p)^r\leqslant \mu^n\kappa(A)^r.
\end{equation}

Since $\hat{B}\in \mathfrak{D}_{\mathcal{C}^\ast}$, there exists $\delta\in \mathcal{C}^\ast$ such that $\kappa(B_{t-s})\leqslant \delta(t-s)$ for all $s\leqslant t$. We claim that for $n\in \mathbb{N}$ and $A\subset B_{t-nT}$ we have
\begin{equation}\label{xcvb}
    \kappa(S(t,t-nT)A)^r\leqslant 2^r\mu^n\kappa(A)^r\leqslant 2^r\mu^n \delta^r(t-nT).
\end{equation}

Assume that the first inequality in \eqref{xcvb} fails. Then we can choose $a>0$ such that 
\[
2^r \mu^n\kappa(A)^r < a< \kappa(S(t,t-nT)A)^r.
\]
Thus implies that
\begin{align*}
\beta(S(t,t-nT)A)\geqslant \frac12 \kappa(S(t,t-nT)A)>\frac{a^{1/r}}{2},
\end{align*}
that is, $S(t,t-nT)A$ has no finite cover of balls of radius less than or equal to $\frac{a^{1/r}}{2}$. Take an arbitrary $x_1\in A$. Then, there exists $x_2\in A$ such that $d(S(t,t-nT)x_1, S(t,t-nT)x_2)>\frac{a^{1/r}}{2}$, for otherwise $S(t,t-nT)A\subset \overline{B}_{\frac{a^{1/r}}{2}}(S(t,t-nT)x_1)$. Following this idea, there exists $x_3\in A$ such that $d(S(t,t-nT)x_3,S(t,t-nT)x_i)>\frac{a^{1/r}}{2}$  for $i=1,2$, for otherwise $S(t,t-nT)A$ would be contained in the union of two balls of radius $\frac{a^{1/r}}{2}$. This process gives us a sequence $\left\{x_i\right\}_{i\in\mathbb{N}}\subset A$ such that $d(S(t,t-nT)x_i,S(t,t-nT)x_j)>\frac{a^{1/r}}{2}$ for all $i\neq j$. Therefore
\begin{align*}
d(S(t,t-nT)x_i,S(t,t-nT)x_j)^r>\frac{a}{2^r}>\mu^n\kappa(A)^r,
\end{align*}
which contradicts \eqref{xcva}. The second inequality of \eqref{xcvb} is trivial.

In particular, we conclude from \eqref{xcvb} that 
\[
\kappa(S(t,t-nT)B_{t-nT})\leqslant 2\delta(t-nT)\mu^{\frac{n}{r}} \quad \hbox{ for } n\in \mathbb{N}.
\]
Since $\delta\in \mathcal{C}^\ast$, there exists $n_0\in \mathbb{N}$ such that $\delta(t-nT)\mu^{\frac{n}{2r}}\leqslant 1$ for $n\geqslant n_0$. Thus, for $n\geqslant n_0$ we obtain
\begin{equation}\label{xcqv}
\kappa(S(t,t-nT)B_{t-nT})\leqslant 2\mu^{\frac{n}{2r}}.
\end{equation}
Finally, for $s\geqslant (n_0+1)T$, if $n\in \mathbb{N}$ is such that $\frac{s}{T}-1<n \leqslant \frac{s}{T}$, we have $n\geqslant n_0$ and $nT\leqslant s$. From the positive invariance of $\hat{B}$ and \eqref{xcqv} we obtain
\begin{align*}
\kappa& (S(t,t-s)B_{t-s})=\kappa\left(S\left(t,t-nT\right)S\left(t-nT, t-s\right)B_{t-s}\right)\\
&\leqslant \kappa\left(S\left(t,t-nT\right)B_{t-nT}\right)\leqslant 2\mu^{\frac{n}{2r}}\leqslant 2\mu^{\frac{s}{2rT}}\mu^{-\frac{1}{2r}}\leqslant C\mu^{\omega s},
\end{align*}
where $C=2\mu^{-\frac{1}{2r}}$ and $\omega=\frac{1}{2rT}$. It follows from Proposition \ref{lemmaxnz} that $S$ is exponentially $\mathfrak{D}_{\mathcal{C}^\ast}$-pullback $\kappa$-dissipative, and therefore, from Proposition \ref{continuouscase} there exists a generalized exponential $\mathfrak{D}_{\mathcal{C}^\ast}$-pullback attractor $\hat{M}$ for $S$ with $\hat{M}\subset \hat{B}$. Since $\hat{B}\in \mathfrak{D}_{\mathcal{C}^\ast}$ and $\mathfrak{D}_{\mathcal{C}^\ast}$ is a universe, we obtain $\hat{M}\in \mathfrak{D}_{\mathcal{C}^\ast}$. 
\end{proof}
\end{theorem}

We point out that this result, although specific to $\mathfrak{D}_{\mathcal{C}^\ast}$, holds true for any universe $\mathfrak{D}$ such that for each $t\in \mathbb{R}$ and $n\in \mathbb{N}$, we have $\kappa(B_{t-nT})\leqslant \delta(t-nT)$, where $\delta$ is a function such that $\mathbb{N}\ni n\mapsto \delta(t-nT)\mu^{\alpha n}$ is bounded for all $n$ sufficiently large, for some $0<\alpha<\frac1r$. In particular, it holds for the universe of \textbf{backwards bounded families}
\[
\mathfrak{D}_{bb} = \{\hat{D}\colon \cup_{s\leqslant t}D_s \hbox{ is bounded for each } t\in \mathbb{R}\},
\]
and for the universe of \textbf{uniformly bounded families}
\[
\mathfrak{D}_{ub} = \{\hat{D}\colon \cup_{t\in \mathbb{R}}D_t \hbox{ is bounded}\}.
\]

\section{Application to a nonautonomous wave equation} \label{Application}

Inspired by the works \cite{yan2023long,ZhaoZhao2020,ZhaoZhong2022} and adapting the results of \cite{Pecorari}, we study the nonautonomous wave equation with non-local weak damping and anti-damping \eqref{ourproblem}, and prove the existence of a \textit{generalized exponential $\mathfrak{D}_{\mathcal{C}^\ast}$-pullback attractor}. 

\subsection{Auxiliary estimates}  Now we present a few estimates, regarding the functions $f$ and $F$, that will help us in the results to come. They are analogous to the ones presented in \cite{Pecorari}, and for this reason, we omit their proofs. To simplify the notation, we also omit the $(\Omega)$ from the subscript of the norms. For instance, we write $\|\cdot\|_{L^2}$ instead of $\|\cdot\|_{L^2(\Omega)}$ and henceforth. 

For $v\in L^2(\Omega)$ we have
\begin{equation*}
\left\|\int_{\Omega}K(x,y)v(y)dy\right\|_{L^2} \leqslant K_0\|v\|_{L^2}.
\end{equation*}

For all $t,v,w\in \mathbb{R}$ we have:
\begin{align}
 & |f(t,v)|\leqslant 2c_0(t)(1+|v|^3);\nonumber\\
 & |f(t,v)-f(t,w)|\leqslant 2c_0(t)(1+|v|^2+|w|^2)|v-w|;\nonumber\\
\label{funcFimp.b} & |F(t,v)|\leqslant 4c_0(t)(1+|v|^4);\\
 & |F(t,v)-F(t,w)|\leqslant 8c_0(t)(1+|v|^3+|w|^3)|v-w|;\nonumber\\
 & \left|\frac{\partial F}{\partial t}(t,v)-\frac{\partial F}{\partial t}(t,w)\right| \leqslant 2c_0(t)|v-w|.\nonumber
\end{align}

In what follows, we use $\hookrightarrow$ to denote continuous inclusions. With the estimates above,  H\"older's inequality, and the continuous inclusion $H^1_0(\Omega)\hookrightarrow L^6(\Omega)$, we can show that there exists a constant $c>0$ such that  for $L_0(t)=c c_0(t)$ we have
\begin{equation}\label{wvtlema}
\|f(t,v)-f(t,w)\|_{L^2}\leqslant L_0(t)(1+\|v\|_{H^1_0}^2+\|w\|_{H^1_0}^2) \|v-w\|_{H^1_0} \quad \hbox{ for all } v,w\in H^1_0(\Omega).
\end{equation}

Also, from the definition of $\liminf$, it follows that given $0<\mu_0<\lambda_1$ there exists $M=M(\mu_0)>0$ such that 
\begin{equation}\label{propax}
\inf_{t\in \mathbb{R}}\frac{\partial f}{\partial v}(t,v)>-\mu_0 \quad \hbox{ and } \quad \inf_{t\in \mathbb{R}}\frac{f(t,v)}{v}>-\mu_0 \quad \hbox{ for all } |v|>M.
\end{equation}

It is also clear that for each $M>0$ we have 
\begin{equation*}
\int_{-M}^M|f(t,v)|dv \leqslant 8(1+M^4)c_0(t) \hbox{ for all } t\in \mathbb{R}.
\end{equation*}

Now, fixing $0<\mu_0<\lambda_1$ and considering $M=M(\mu_0)>0$ given by \eqref{propax}, we obtain
\begin{equation}\label{propF}
F(t,v)\geqslant -\tfrac{\mu_0+\lambda_1}{4}v^2-8(1+M^4)c_0(t) \quad \hbox{ for } |v|>M \hbox{ and } t\in \mathbb{R},
\end{equation}
and 
\begin{equation}\label{propF1}
|F(t,v)|\leqslant 8(1+M^4)c_0(t) \quad \hbox{ for } |v|\leqslant M \hbox{ and } t\in \mathbb{R}.
\end{equation}
Finally, we can prove that there exists a constant $e_0>0$ such that
\begin{equation*}
F(t,v)\leqslant vf(t,v)+\frac{\mu_0}{2}v^2+e_0 \quad \hbox{ for } |v|>M \hbox{ and } t\in\mathbb{R}.
\end{equation*}

\bigskip \noindent \textbf{Translation of the problem.} \bigskip

Fixing $s\in \mathbb{R}$ and setting $v(t,x):=u(t+s,x)$ for $t\geqslant 0$ and $x\in\Omega$, we formally have
\[
\begin{aligned}
&v_{tt}(t,x)-\Delta v(t,x)+k_s(t)v_t(t,x)+f_s(t, v(t,x)) -\int_{\Omega}K(x,y)v_t(t,y)dy-h( x)\\
&=u_{tt}(t+s, x)-\Delta u(t+s, x)+k_s(t)u_t(t+s, x)+f_s(t,u(t+s,x))-\int_{\Omega}K(x,y)u_t(t+s,y)dy-h(x)=0,
\end{aligned}
\]
where the boundary and initial conditions become
\begin{align*}
& v(t,x) = u(t+s,x) = 0 \quad \hbox{ for } (t,x)\in \left[0,\infty \right) \times \partial\Omega,\\
& v(0,x)=u(s,x)=u_0(x), \ v_t(0,x)=u_t(s,x)=u_1(x) \quad \hbox { for } x\in \Omega.
\end{align*}
Thus, we will study the \textit{translated} boundary and initial conditions problem
\begin{equation}\label{eqondap}\tag{tNWE}
\left\{\begin{aligned}
&v_{tt}(t,x)\!-\!\Delta v(t,x)\!+\!k_s(t)v_t(t,x)\!+\!f_s(t,v(t,x)) \!=\! \int_{\Omega}K(x,y)v_t(t,y)dy\! +\! h(x), \ (t,x)\!\in\! [0,\infty)\!\times\!\Omega,\\
& v(t,x) = 0, \ (t,x)\in [0,\infty)\times \partial \Omega,\\
& v(0,x)=u_0(x), \ v_t(0,x) = u_1(x), \ x\in \Omega,
\end{aligned}\right.
\end{equation}
instead of \eqref{ourproblem}. This problem is equivalent to the initial one, but with the nonautonomous terms being $k_s(\cdot)=k(\cdot+s)$ and $f_s(\cdot,\cdot)=f(\cdot+s,\cdot)$ instead of $k$ and $f$. Additionally, we denote $F_s(\cdot , \cdot)=F(\cdot + s, \cdot)$.

\newcommand{\vetor}[2]{\left[\begin{smallmatrix} #1 \\ #2 \end{smallmatrix}\right]}

Taking $w=v_t$ in \eqref{eqondap}, we obtain 
\[
w_t=v_{tt}=\Delta v-k_s(t)v_t-f_s(t,v)+\int_{\Omega}K(x,y)v_t(t,y)dy+h(x),
\]
and thus
\begin{equation*}
\begin{aligned}
\frac{d}{dt}\begin{bmatrix}v\\w\end{bmatrix}&=\begin{bmatrix}v_t\\w_t\end{bmatrix}=
\begin{bmatrix} w \\ \Delta v-k_s(t)v_t -f_s(t,v)+\int_{\Omega}K(x,y)v_t(t,y)dy+h(x)\end{bmatrix}\\
&=\begin{bmatrix} 0 & I \\ \Delta & 0 \end{bmatrix}\begin{bmatrix}v\\w\end{bmatrix}+\begin{bmatrix} 0 \\ \int_{\Omega}K(x,y)w(t,y)dy+h(x)-f_s(t,v)-k_s(t)w \end{bmatrix}.
\end{aligned}
\end{equation*}
Setting $X:=H^1_0(\Omega)\times L^2(\Omega)$, $V=\vetor{v}{w}$, $V_0=\vetor{u_0}{u_1}$, $\mathcal{A}=\left[\begin{smallmatrix} 0 & I \\ \Delta & 0 \end{smallmatrix}\right]$ and $\mathcal{G}_s(t,V)=\vetor{0}{G_s(t,V)}$, where 
\[
G_s(t,V)=\int_{\Omega}K(x,y)w(t,y)dy+h(x)-f_s(t,v)-k_s(t)w,
\]
we can represent \eqref{eqondap} by an abstract Cauchy problem
\begin{equation}\label{Abs.Ev.Equation}\tag{ACP}
\left\{\begin{aligned}
    &\frac{dV}{dt}=\mathcal{A}V+\mathcal{G}_s(t,V), \ t>0\\
    &V(0)=V_0\in X.
\end{aligned} \right.
\end{equation}
Clearly $X$ is a Hilbert space with the inner product defined by          
\begin{equation*}
\left\langle \vetor{v_1}{w_1},\vetor{v_2}{w_2} \right \rangle_X=\langle v_1,v_2\rangle_{H^1_0(\Omega)}+\langle w_1, w_2 \rangle_{L^2(\Omega)},
\end{equation*}
with associate norm
\begin{equation*}
\left\|\vetor{v}{w}\right\|_X^2= \|v\|^2_{H^1_0} + \| w\|^2_{L^2}.
\end{equation*}

Using Lumer-Phillips' Theorem and the results of \cite{Pazy} for semilinear evolution equations, we obtain the following:

\begin{proposition}\label{solutions.Ex}
Given $s\in \mathbb{R}$ and $V_0\in X$ there exists a unique maximal weak solution $V(\cdot,s,V_0)\colon [0,\tau_m)\to X$ of \eqref{Abs.Ev.Equation}, that is, a continuous function such that
\[
V(t,s,V_0) = e^{\mathcal{A}t}V_0 + \int_0^t e^{\mathcal{A}(t-\tau)} \mathcal{G}_s(\tau,V(\tau,s,V_0))d\tau \quad \hbox{ for } t\in [0,\tau_m),
\]
such that either $\tau_m=\infty$, or $\tau_m<\infty$ and $\limsup_{t\to \tau_m^-}\|V(t,s,V_0)\|_X = \infty$, where $\{e^{\mathcal{A}t}\colon t\geqslant 0\}$ is the $C^0$-semigroup of contractions generated by $\mathcal{A}$ in $X$.
\end{proposition}

Writing $V(t,s,V_0)=\vetor{v(t)}{w(t)}$ we can see that
\[
v \in C([0,\tau_m),H^1_0(\Omega)) \hbox{ and } w\in C([0,\tau_m),L^2(\Omega)).
\]
Since $w=v_t$, we conclude that
\[
v\in C([0,\tau_m),H^1_0(\Omega))\cap C^1([0,\tau_m),L^2(\Omega))
\]
is the unique maximal weak solution to \eqref{eqondap}. Furthermore, $v_{tt}\colon (0,\tau_m)\to H^{-1}(\Omega)$ is well-defined and
\begin{equation*}
\|v_{tt}(t)\|_{H^{-1}}  \leqslant \|v(t)\|_{H^1_0} + k_s(t) \lambda_1^{-\frac12}\|v_t(t)\|_{L^2} + \lambda_1^{-\frac12}\|f_s(t,v(t))\|_{L^2} +K_0\lambda_1^{-\frac12}\|v_t(t)\|_{L^2} + \lambda_1^{-\frac12}h_0.
\end{equation*}

Later we show that $\tau_m=\infty$ for all $s\in \mathbb{R}$ and $V_0\in X$ (see Remark \ref{remark.Global}). For $s\in \mathbb{R}$ and $V_0:=(u_0,u_1)\in X$, defining 
\begin{equation*}
S(t,s)(u_0,u_1) = V(t-s,s,V_0) \quad \hbox{ for } t\geqslant s,
\end{equation*}
setting $(u(t),y(t)) = S(t,s)(u_0,u_1)$, then $u(t)=v(t-s)$ and $y(t)=v_t(t-s)=u_t(t)$. Thus $(u(t),u_t(t))$ is the weak solution of \eqref{ourproblem}, and $S=\{S(t,s)\colon t\geqslant s\}$ defines an evolution process in $X$ associated with \eqref{ourproblem}, provided we have continuity with respect to initial data (which we also show, see Theorem \ref{Lipschitz}).

\subsection{Properties of uniform pullback absorption} \label{absorbingfamily}

In this section we prove the existence of a closed, positively invariant and bounded pullback absorbing family. Recall that $0<\mu_0<\lambda_1$ is fixed and we consider $M=M(\mu_0)>0$ given by \eqref{propax}. For $v\in L^2(\Omega)$ we define
\[
\Omega_1:=\{x\in\Omega\colon |v(x)|>M\} \quad \hbox{ and } \quad \Omega_2:=\{x\in\Omega\colon  |v(x)|\leqslant M\}.
\]
From \eqref{propF} and \eqref{propF1}, for $t\in \mathbb{R}$ we obtain
\[
\begin{aligned}
\int_{\Omega}& F(t,v)dx = \int_{\Omega_1}F(t,v)dx + \int_{\Omega_2}F(t,v)dx\\
&\geqslant -\int_{\Omega_1} \left[\tfrac{\mu_0+\lambda_1}{4} |v|^2+ 8c_0(t)(1+M^4)\right]dx - \int_{\Omega_2} 8c_0(t)(1+M^4) dx \geqslant -\tfrac{\mu_0+\lambda_1}{4}\|v\|^2_{L^2} - C_0(t), 
\end{aligned}
\]
where $C_0(t):=8(1+M^4)|\Omega|c_0(t)$. This means that
\begin{equation}\label{observaa}
\int_{\Omega} F(t,v)dx + \tfrac{\mu_0+\lambda_1}{4}\|v\|^2_{L^2}+C_0(t)\geqslant 0,
\end{equation}
and it is clear that $C_0\in {\mathcal{C}^\ast}$.

For each $V_0=(u_0,u_1)\in X$, we define the function $E_s(\cdot,V_0)\colon \mathbb{R}^+\to \mathbb{R}$ by
\begin{equation}\label{energyy}
 E_s(t,V_0)=\tfrac{1}{2}(\|v_t\|_{L^2}^2+\|v\|_{H^1_0}^2) + \int_{\Omega}F_s(t,v)dx+\tfrac{\lambda_1+\mu_0}{4}\|v\|^2_{L^2}+C_0(t+s),
\end{equation}
where $V(t,s,V_0)=(v(t),v_t(t))$ for $t\geqslant 0$. For simplicity, we write $E_s(t)$ instead of $E_s(t,V_0)$, but we must keep in mind that this function depends on $V_0$, since it depends on the solution $V(t,s,V_0)$. Note that for $t\geqslant s$ and $V_0=(u_0,u_1)\in X$, from \eqref{energyy} and \eqref{observaa} we have
\begin{align*}
\|S(t,s)(u_0,u_1)\|_X^2&=\|V(t-s,s,V_0)\|_X^2 \leqslant 2E_s(t-s).
\end{align*}
Hence, the study of the function $E_s$ is paramount, since this function bounds the norm of $S(t,s)(u_0,u_1)$ for $t\geqslant s$.

For $\varepsilon>0$ and $(v,v_t)=V(t,s,V_0)$ we define an auxiliary function $\mathcal{V}_\varepsilon^s(\cdot,V_0)\colon \mathbb{R}^+\to \mathbb{R}$ by
\begin{equation*}
\mathcal{V}_\varepsilon^s(t,V_0)=\tfrac{1}{2}\big(\|v\|^2_{H^1_0} + \|v_t\|^2_{L^2}\big)+\int_{\Omega}F_s(t,v)dx-\int_{\Omega}hvdx+\varepsilon\int_{\Omega}v_tvdx.
\end{equation*}
As we did for the function $E_s$, to simplify the notation we simply write $\mathcal{V}_\varepsilon^s(t)$ instead of $\mathcal{V}_\varepsilon^s(t,V_0)$. Adapting the results of \cite{Pecorari}, we obtain the following proposition. 

\begin{proposition}\label{propa401} 
For all $s\in \mathbb{R}$ and $t\geqslant 0$ we have
\begin{align*}
-\int_{\Omega}f_s(t,v)vdx&\leqslant -\int_{\Omega}F_s(t, v)dx - \tfrac{\lambda_1+\mu_0}{4}\|v\|^2_{L^2} - C_0(t+s)+\tfrac{1}{4}\left(\tfrac{3\mu_0}{\lambda_1}+1\right)\|v\|^2_{H^1_0}+g_0(t+s),
\end{align*}
where $g_0(t):=e_0|\Omega_1|+8c_0(t)(1+M^{4})|\Omega_2|+8Mc_0(t)(1+M^{4})+C_0(t)$ for all $t\in \mathbb{R}$. Furthermore, there exists a constant $K^\ast>0$ such that for all $s\in \mathbb{R}$, $t\geqslant 0$ and $\varepsilon>0$ we have
\begin{equation*}\label{propa303.a}
\begin{aligned}
& -k_s(t)\|v_t\|_{L^2}^{2}-\varepsilon k_s(t)\int_{\Omega}v_t vdx \leqslant -k_0\|v_t\|_{L^2}^{2}\left(1-\tfrac{\varepsilon K^\ast k_1}{k_0}\right)+\tfrac{\varepsilon}{12}\left(1-\tfrac{\mu_0}{\lambda_1}\right)\|v\|_{H^1_0}^2; 
\end{aligned}
\end{equation*}
Also, for all $s\in \mathbb{R}$, $t\geqslant 0$ and $\varepsilon>0$ we have
\begin{align*}
& \left|\int_{\Omega\times\Omega}K(x,y)v_t(t,y)v(t,x)dydx \right|\leqslant \tfrac{1}{12}\left(1-\tfrac{\mu_0}{\lambda_1}\right)\|v\|_{H^1_0}^2+\tfrac{3K_0^2}{\lambda_1-\mu_0}\|v_t\|^2_{L^2};\\
& \left|\int_{\Omega \times \Omega}K(x,y)v_t(t,y)v_t(t,x)  dydx\right| \leqslant K_0\|v_t\|^2_{L^2}.\\
& \left|\int_{\Omega}hv dx\right|\leqslant \tfrac{1}{16}\left(1-\tfrac{\mu_0}{\lambda_1}\right)\|v\|^2_{H^1_0}+\tfrac{4}{\lambda_1-\mu_0}h_0^2.
\end{align*}    
Finally, there exists $\varepsilon_0>0$ such that for all $0<\varepsilon \leqslant\varepsilon_0$, $s\in \mathbb{R}$ and $t\geqslant 0$ we have
\[
\tfrac14\left(1-\tfrac{\mu_0}{\lambda_1}\right) E_s(t) - d_0(t+s) \leqslant \mathcal{V}_\varepsilon^s(t) \leqslant \tfrac54 E_s(t) + d_0(t+s),
\]
where $d_0(t):=C_0(t)+\frac{4}{\lambda_1-\mu_0}h_0^2$ for each $t\in \mathbb{R}$.
\end{proposition}

Note that the functions $g_0$ and $d_0$ mentioned in the previous proposition are in ${\mathcal{C}^\ast}$. Formally multiplying \eqref{eqondap} by $v_t+\varepsilon v$ and integrating over $\Omega$ we can show that
\begin{equation}\label{vepsilont}
\begin{aligned}
\frac{d}{dt}& \mathcal{V}_\varepsilon^s(t)=-k_s(t)\|v_t\|^{2}_{L^2}-\varepsilon k_s(t)\int_{\Omega}v_tv dx +\int_{\Omega}\frac{\partial F_s}{\partial t}(t, v) dx \\ 
&\qquad +\int_{\Omega \times \Omega}K(x,y)v_t(t,y)v_t(t,x)dydx  -\varepsilon\|v\|^2_{H^1_0}+\varepsilon\|v_t\|^2_{L^2}-\varepsilon\int_{\Omega}f_s(t, v)vdx \\
&\qquad +\varepsilon \int_{\Omega \times \Omega} K(x,y)v_t(t,y)v(t,x)dydx+\varepsilon \int_{\Omega}hvdx.
\end{aligned}
\end{equation}

\begin{lemma}
There exists $\varepsilon_1\in (0,\varepsilon_0]$ such that for all $\varepsilon\in (0,\varepsilon_1]$, $s\in\mathbb{R}$ and $t\geqslant 0$ we have
    \begin{equation*}
    \frac{d}{dt}\mathcal{V}_\varepsilon^s(t)\leqslant -\frac{4}{5}\left(1-\frac{\mu_0}{\lambda_1}\right)\varepsilon \mathcal{V}_\varepsilon^s(t)+\delta_1(t+s),
    \end{equation*}
where $\delta_1(t):=\tfrac{4}{5}\varepsilon_0\big(1-\tfrac{\mu_0}{\lambda_1}\big)d_0(t)+c_0(t)+\varepsilon_0g_0(t) + \tfrac{4\varepsilon_0}{\lambda_1-\mu_0}h_0^2$ for each $t\in \mathbb{R}.$ Furthermore, for every $s\in\mathbb{R}$ and $t\geqslant 0$ we have
\[
\mathcal{V}_\varepsilon^s(t)\leqslant \mathcal{V}_\varepsilon^s(0)e^{-\beta \varepsilon t}+\int_0^t\delta_1(r+s)e^{-\beta \varepsilon(t-r)}dr,
\]
where $\beta:=\frac45(1-\frac{\mu_0}{\lambda_1})$.
\end{lemma}
\begin{proof}
    Since, from \ref{cond5}, $\int_\Omega \frac{\partial F_s}{\partial t}(t,x)dx \leqslant c_0(t+s)$, from \eqref{vepsilont} and Proposition \ref{propa401} we obtain
\begin{equation*}
\begin{aligned}
& \frac{d}{dt} \mathcal{V}_\varepsilon^s(t)\leqslant -k_0\|v_t\|_{L^2}^{2}\big(1-\tfrac{\varepsilon K^\ast k_1}{k_0} \big)+\tfrac{\varepsilon}{12}\left(1-\tfrac{\mu_0}{\lambda_1}\right)\|v\|_{H^1_0}^2+c_0(t+s)+K_0\|v_t\|^2_{L^2}-\varepsilon\|v\|^2_{H^1_0}\\ 
&\quad +\varepsilon\|v_t\|^2_{L^2}-\varepsilon\int_{\Omega}F_s(t, v)dx - \varepsilon\tfrac{\lambda_1+\mu_0}{4}\|v\|^2_{L^2} -\varepsilon C_0(t+s) +\tfrac{\varepsilon}{4}\left(\tfrac{3\mu_0}{\lambda_1}+1\right)\|v\|^2_{H^1_0}\\
&\quad+\varepsilon g_0(t+s) + \tfrac{\varepsilon}{12}\left(1-\tfrac{\mu_0}{\lambda_1}\right)\|v\|_{H^1_0}^2 +\tfrac{3\varepsilon K_0^2}{\lambda_1-\mu_0}\|v_t\|^2_{L^2}+\tfrac{\varepsilon}{16}\left(1-\tfrac{\mu_0}{\lambda_1}\right)\|v\|^2_{H^1_0}+\tfrac{4\varepsilon}{\lambda_1-\mu_0}h_0^2  \\
&\stackrel{(1)}{\leqslant} -k_0\|v_t\|_{L^2}^{2}\left[1-\tfrac{K_0}{k_0}-\varepsilon\left(\tfrac{K^\ast k_1}{k_0}+\tfrac{1}{k_0}+\tfrac{3K_0^2}{(\lambda_1-\mu_0)k_0}+\tfrac{\lambda_1-\mu_0}{2k_0\lambda_1}\right)\right]- \tfrac{25\varepsilon}{48}\left(1-\tfrac{\mu_0}{\lambda_1}\right)\|v\|^2_{H^1_0} \\
&\quad -\tfrac{\varepsilon}{2}\left(1-\tfrac{\mu_0}{\lambda_1}\right)\|v_t\|^2_{L^2}-\varepsilon\int_{\Omega}F_s(t,v)dx-\varepsilon\tfrac{\lambda_1+\mu_0}{4}\|v\|^2_{L^2}\\
&\quad+c_0(t+s)- \varepsilon C_0(t+s) + \varepsilon_0 g_0(t+s) + \tfrac{4\varepsilon_0}{\lambda_1-\mu_0}h_0^2\\
&\stackrel{(2)}{\leqslant} -\tfrac{\varepsilon}{2}\left(1-\tfrac{\mu_0}{\lambda_1}\right)(\|v_t \|^2_{L^2} + \|v\|^2_{H^1_0})-\varepsilon\Big[\int_{\Omega}F_s(t,v)dx+\tfrac{\lambda_1+\mu_0}{4}\|v\|_{L^2}^2 + C_0(t+s)\Big]\\
&\quad+c_0(t+s)+\varepsilon_0g_0(t+s) + \tfrac{4\varepsilon_0}{\lambda_1-\mu_0}h_0^2\\
&\leqslant -\tfrac{4}{5}\big(1-\tfrac{\mu_0}{\lambda_1}\big)\varepsilon \mathcal{V}_\varepsilon^s(t)+\tfrac{4}{5}\varepsilon_0\big(1-\tfrac{\mu_0}{\lambda_1}\big)d_0(t+s)+c_0(t+s)+\varepsilon_0g_0(t+s) + \tfrac{4\varepsilon_0}{\lambda_1-\mu_0}h_0^2 
\end{aligned}
\end{equation*}
where in $(1)$ we added and subtracted the term $\tfrac{\varepsilon}{2}(1-\tfrac{\mu_0}{\lambda_1})\|v_t\|^2_{L^2}$ and used Poincar\'e's inequality. In (2) we choose $\varepsilon_1 \in (0,\varepsilon_0]$ such that 
\begin{equation}\label{epsiloncond}
\varepsilon\left(\tfrac{K^\ast k_1}{k_0}+\tfrac{1}{k_0}+\tfrac{3K_0^2}{(\lambda_1-\mu_0)k_0}+\tfrac{\lambda_1-\mu_0}{2k_0\lambda_1}\right)\leqslant 1-\tfrac{K_0}{k_0} \quad \hbox{ for all } \varepsilon \in (0,\varepsilon_1],
\end{equation}
which is possible since $0\leqslant K_0<k_0$.

The last claim follows directly from Gronwall's inequality, and the proof is complete.
\end{proof}

Note that the function $\delta_1$ is in $\mathcal{C}^\ast$.

\begin{remark}[Solutions to \eqref{Abs.Ev.Equation} are global] \label{remark.Global}
For $s\in \mathbb{R}$ and $V_0\in X$, the arguments up until now work for $t\in [0,\tau_m)$, and we obtain
\[
\mathcal{V}_\varepsilon^s(t) \leqslant \mathcal{V}_\varepsilon^s(0)e^{-\beta \varepsilon t} + \int_0^t \delta_1(r+s)e^{-\beta \varepsilon (t-r)}dr \quad \hbox{ for } t\in [0,\tau_m).
\]
From \eqref{funcFimp.b} and the continuous inclusion $H^1_0(\Omega)\hookrightarrow L^4(\Omega)$, with constant $c>0$, we obtain 
\begin{align*}
\Big|\int_{\Omega}&F(s, v) dx\Big| \leqslant \int_{\Omega}|F(s, v)| dx \leqslant \int_{\Omega}4c_0(s)(1+|v|^4) dx\\
& =4c_0(s)|\Omega| + 4c_0(s)\|v\|_{L^4}^{4} \leqslant 4c_0(s)|\Omega|+4c_0(s)c^4\|v\|_{H^1_0}^4  \leqslant \alpha_0(s)(1+\left\|v\right\|^4_{H^1_0}),    
\end{align*}
where $\alpha_0(s)=4\max\{|\Omega|,c^4\}c_0(s)$. Thus 
\begin{align*}
E_s(0)&=\tfrac{1}{2}\left(\left\|u_1\right\|_{L^2}^2+\| u_0\|_{H^1_0}^2\right)+\int_{\Omega}F(s, u_0)dx+\tfrac{\lambda_1+\mu_0}{4}\|u_0\|^2_{L^2} +C_0(s)\\
&\leqslant \tfrac{1}{2}\|V_0\|^2_X+\alpha_0(s)(1+\| u_0\|^4_{H^1_0})+\tfrac{\lambda_1+\mu_0}{4\lambda_1}\|u_0\|^2_{H^1_0}+C_0(s)\\
& \leqslant \Big(\tfrac12+\tfrac{\lambda_1+\mu_0}{2\lambda_1}\Big)\|V_0\|^2_X + \alpha_0(s)(1+\|V_0\|^4_X) + C_0(s).
\end{align*}
Using Proposition \ref{propa401} we obtain
\begin{align*}
\mathcal{V}_\varepsilon^s(0)&\leqslant \tfrac{5}{4}E_s(0)+d_0(s)\\
&\leqslant \tfrac54\Big(\tfrac12+\tfrac{\lambda_1+\mu_0}{2\lambda_1}\Big)\|V_0\|^2_X + \tfrac54\alpha_0(s)(1+\|V_0\|^4_X) + \tfrac54C_0(s)+d_0(s).
\end{align*}
Again, for $t\in [0,\tau_m)$, Proposition \ref{propa401} gives us
\begin{equation}
\begin{aligned}\label{cr2cr1}
\|V(t,s,V_0)\|^2_X & \leqslant 2E_s(t) \leqslant \tfrac{8\lambda_1}{\lambda_1-\mu_0}\mathcal{V}_\varepsilon^s(t) + \tfrac{8\lambda_1}{\lambda_1-\mu_0}d_0(t+s)\\
& \leqslant \tfrac{8\lambda_1}{\lambda_1-\mu_0}\Big[\tfrac54\Big(\tfrac12+\tfrac{\lambda_1+\mu_0}{2\lambda_1}\Big)\|V_0\|^2_X + \tfrac54\alpha_0(s)(1+\|V_0\|^4_X) + \tfrac54C_0(s)+d_0(s)\Big]e^{-\beta \varepsilon t}\\
& \qquad + \tfrac{8\lambda_1}{\lambda_1-\mu_0}\int_0^t \delta_1(r+s)e^{-\beta \varepsilon (t-r)}dr +  \tfrac{8\lambda_1}{\lambda_1-\mu_0}d_0(t+s).
\end{aligned}
\end{equation}
Making $t\to \tau_m^-$ we see that $\limsup_{t\to \tau_m^-}\|V(t,s,V_0)\|_X<\infty$, which implies that $\tau_m=\infty$. 
\end{remark}

With these considerations, we state and prove the main result of this section.

\begin{theorem}[Existence of a pullback absorbing family]\label{teo.existence.paf}
Given $\hat{D}\in \mathfrak{D}_{\mathcal{C}^\ast}$, there exists $\Gamma\in {\mathcal{C}^\ast}$ such that for all $t\in \mathbb{R}$, $\tau\geqslant 0$ and $V_0=(u_0,u_1)\in D_{t-\tau}$ we have
\[
\|S(t,t-\tau)(u_0,u_1)\|^2_X \leqslant \Gamma(t-\tau)e^{-\beta \varepsilon \tau} +  \tfrac{8\lambda_1}{\lambda_1-\mu_0}\int_0^\infty \delta_1(t-u)e^{-\beta \varepsilon u}du +  \tfrac{8\lambda_1}{\lambda_1-\mu_0}d_0(t).
\]
Furthermore, defining
\[
r_0(\xi):= \left\{\tfrac{8\lambda_1}{\lambda_1-\mu_0}\int_0^\infty \delta_1(\xi-u)e^{-\beta \varepsilon u}du +  \tfrac{8\lambda_1}{\lambda_1-\mu_0}d_0(\xi)+1\right\}^\frac12 \quad \hbox{ for } \xi \in \mathbb{R},
\]
and $B_\xi = \overline{B}^X_{r_0(\xi)}(0)$ for each $\xi\in \mathbb{R}$, the family $\hat{B}=\{B_t\}_{t\in \mathbb{R}}\in \mathfrak{D}_{\mathcal{C}^\ast}$ is uniformly $\mathfrak{D}_{\mathcal{C}^\ast}$-pullback absorbing for $S$.
\end{theorem}
\begin{proof}
Let $t\in \mathbb{R}$, $\tau\geqslant 0$, $\hat{D}\in \mathfrak{D}_{\mathcal{C}^\ast}$ and $V_0\in D_{t-\tau}$. If $r\in {\mathcal{C}^\ast}$ is such that $D_\xi \subset \overline{B}_{r(\xi)}(0)$ for all $\xi\in \mathbb{R}$, from \eqref{cr2cr1} we obtain
\begin{align*}
\|V(\tau,t-\tau,V_0)\|^2_X & \leqslant \Gamma(t-\tau) e^{-\beta \varepsilon \tau} + \tfrac{8\lambda_1}{\lambda_1-\mu_0}\int_0^\tau \delta_1(r+t-\tau)e^{-\beta \varepsilon (\tau-r)}dr +  \tfrac{8\lambda_1}{\lambda_1-\mu_0}d_0(t)\\
& \leqslant \Gamma(t-\tau) e^{-\beta \varepsilon \tau} + \tfrac{8\lambda_1}{\lambda_1-\mu_0}\int_0^\infty \delta_1(t-u)e^{-\beta \varepsilon u}du +  \tfrac{8\lambda_1}{\lambda_1-\mu_0}d_0(t), 
\end{align*}
where for $\xi\in \mathbb{R}$:
\[
\Gamma(\xi):=\tfrac{8\lambda_1}{\lambda_1-\mu_0}\Big[\tfrac54\Big(\tfrac12+\tfrac{\lambda_1+\mu_0}{2\lambda_1}\Big)r(\xi)^2 + \tfrac54\alpha_0(\xi)(1+r(\xi)^4) + \tfrac54C_0(\xi)+d_0(\xi)\Big],
\]
and it is clear that $\Gamma \in {\mathcal{C}^\ast}$. Since $\Gamma\in {\mathcal{C}^\ast}$, there exists $\tau_0 = \tau_0(\hat{D},t)>0$ such that $\sup_{s\leqslant t}\Gamma(s-\tau)e^{-\beta \varepsilon \tau}\leqslant 1$. Thus, for all $s\leqslant t$ and $\tau\geqslant \tau_0$ we have
\[
\|S(s,s-\tau)(u_0,u_1)\|^2_X \leqslant 1+\tfrac{8\lambda_1}{\lambda_1-\mu_0}\int_0^\infty \delta_1(s-u)e^{-\beta \varepsilon u}du +  \tfrac{8\lambda_1}{\lambda_1-\mu_0}d_0(s) = r_0(s)^2,
\]
that is, $S(s,s-\tau)D_{s-\tau}\subset \overline{B}^X_{r_0(s)}(0)=B_s$. This proves that $\hat{B}$ is uniformly $\mathfrak{D}_{\mathcal{C}^\ast}$-pullback absorbing for $S$.

Lastly, using Lemma \ref{integralemC*} we conclude that $r_0^2\in \mathcal{C}^\ast$, which implies that $r_0\in \mathcal{C}^\ast$. Hence $\hat{B}\in \mathfrak{D}_{\mathcal{C}^\ast}$, and the proof is complete.
\end{proof}

\bigskip \noindent \textbf{Existence of a closed, positively invariant and uniformly pullback absorbing family in $\mathfrak{D}_{\mathcal{C}^\ast}$} \label{obtainingC}\bigskip

Using the family $\hat{B}$ presented above, we will construct a family $\hat{C}\in\mathfrak{D}_{\mathcal{C}^\ast}$ which is closed, positively invariant and uniformly $\mathfrak{D}_{\mathcal{C}^\ast}$-pullback absorbing for $S$. Since $\hat{B}\in \mathfrak{D}_{\mathcal{C}^\ast}$, given $t\in \mathbb{R}$, there exists $\tau_1\geqslant 0$ such that $S(s,s-\tau)B_{s-\tau}\subset B_s$ for all $\tau\geqslant \tau_1$ and $s\leqslant t$. Consider the family $\hat{C}=\left\{C_t\right\}_{t\in\mathbb{R}}$ defined by 
\begin{equation}\label{eq.def.C}
    C_t=\overline{\bigcup_{\tau\geqslant \tau_1}S(t,t-\tau)B_{t-\tau}} \quad \hbox{ for each } t\in \mathbb{R}.
\end{equation}

\begin{theorem}
    The family $\hat{C}$ is closed, positively invariant, uniformly $\mathfrak{D}_{\mathcal{C}^\ast}$-pullback absorbing for $S$, and $\hat{C}\in \mathfrak{D}_{\mathcal{C}^\ast}$.
\end{theorem}
\begin{proof}
Clearly $\hat{C}$ is closed, and $C_t \subset B_t$ for all $t\in\mathbb{R}$, which implies that $\hat{C}\in \mathfrak{D}_{\mathcal{C}^\ast}$. Now, let $t\geqslant s$. Then we can write $s=t-\sigma$ for some $\sigma\geqslant 0$. Note that
\begin{align*}   
S(t,s)C_s&=S(t,t-\sigma)C_{t-\sigma}=S(t,t-\sigma)\overline{\bigcup_{\tau\geqslant \tau_1}S(t-\sigma,t-\sigma-\tau)B_{t-\sigma-\tau}} \\
&\subset \overline{S(t,t-\sigma)\bigcup_{\tau\geqslant \tau_1}S(t-\sigma,t-\sigma-\tau)B_{t-\sigma-\tau}}=\overline{\bigcup_{\tau\geqslant \tau_1}S(t,t-\sigma)S(t-\sigma,t-\sigma-\tau)B_{t-\sigma-\tau}}\\
&=\overline{\bigcup_{\tau\geqslant \tau_1}S(t,t-(\sigma+\tau))B_{t-(\sigma+\tau)}} \subset \overline{\bigcup_{\tau\geqslant \sigma+\tau_1}S(t,t-\tau)B_{t-\tau}}\subset \overline{\bigcup_{\tau\geqslant \tau_1}S(t,t-\tau)B_{t-\tau}}=C_t,
\end{align*}
which prove the positive invariance of the family $\hat{C}$.

To see that $\hat{C}$ is uniformly $\mathfrak{D}_{\mathcal{C}^\ast}$-pullback absorbing, let $\hat{D}\in \mathfrak{D}_{\mathcal{C}^\ast}$. We know that given $t\in \mathbb{R}$, there exist $\tau_0, \tau_1 \geqslant 0$ such that $S(s,s-\tau)D_{s-\tau}\subset B_s$ for all $s\leqslant t$ and $\tau\geqslant \tau_0$ and $S(s,s-\tau)B_{s-\tau}\subset B_s$ for all $s\leqslant t$ and $\tau\geqslant \tau_1$. If $s\leqslant t$ and $\tau \geqslant \tau_0+\tau_1$ then $\tau-\tau_1\geqslant \tau_0$ and we obtain
\begin{align*}
S(s,s-\tau)D_{s-\tau} &= S(s,s-\tau_1)S(s-\tau_1,s-\tau)D_{s-\tau}\\
&=S(s,s-\tau_1)S(s-\tau_1,s-\tau_1-(\tau-\tau_1))D_{s-\tau_1-(\tau-\tau_1)} \subset S(s,s-\tau_1)B_{s-\tau_1}\subset C_s,
\end{align*}
and the proof is complete.
\end{proof}

\bigskip \noindent \textbf{Continuous dependence on initial data} \label{continuousdependence}\bigskip

Let $\hat{D}\in \mathfrak{D}_{\mathcal{C}^\ast}$ and $V_1,V_2\in D_s$. Consider $v, w$ the corresponding solutions for \eqref{eqondap} related respectively to these initial data. From Theorem \ref{teo.existence.paf} there exists $\Gamma\in \mathcal{C}^\ast$ such that for all $\tau\geqslant 0$ and $i=1,2$ we have
\begin{equation}\label{uhbijnx}
\|V(\tau,s,V_i)\|_X^2\leqslant \Gamma(s)e^{-\beta \varepsilon \tau} + r_0^2(s+\tau)
\end{equation}
If $V(\tau,s,V_1)=(v(\tau),v_t(\tau))$ and $V(\tau,s,V_2)=(w(\tau),w_t(\tau))$, setting $z:=v-w$, formally we obtain 
\begin{equation}\label{fghjmmm} 
z_{tt}-\Delta z+k_s(\tau)z_t + f_s(\tau,v)-f_s(\tau,w)=\int_{\Omega}K(x,y)z_t(\tau,y)dy.
\end{equation}

If $Z(\tau)=V(\tau,s,V_1)-V(\tau,s,V_2)$, then $Z(\tau)=(z(\tau),z_t(\tau))$ for each $t\geqslant 0$. Multiplying formally \eqref{fghjmmm} by $z_t$ and integrating over $\Omega$, we obtain
\begin{align*}
\frac{d}{d\tau}\|Z(\tau)\|^2_X & = 2\int_{\Omega\times\Omega} K(x,y)z_t(\tau,y)z_t(\tau,x)dydx-2\langle f_s(\tau,v)\!-\!f_s(\tau,w), z_t \rangle -2k_s(\tau) \left\| z_t\right\|_{L^2}^2\\
&\leqslant 2\int_{\Omega\times\Omega}K(x,y)z_t(\tau,y)z_t(\tau,x) dydx-2\langle f_s(\tau,v)-f_s(\tau,w), z_t \rangle.
\end{align*}
Since
\[
\int_{\Omega\times\Omega}K(x,y)z_t(\tau,y)z_t(\tau,x) dydx \leqslant K_0\|z_t\|^2_{L^2},
\]
and, by \eqref{wvtlema} and \eqref{uhbijnx}, we have
\begin{equation*}
\begin{aligned}
-\langle & f_s(\tau, v)-f_s(\tau,w), z_t \rangle \leqslant \|f_s(\tau,v)-f_s(\tau,w)\|_{L^2}\|z_t\|_{L^2}\\
&\leqslant L_0(\tau+s)(1+\|v(\tau)\|_{H^1_0}^{2}+\|w(\tau)\|_{H^1_0}^{2})\|v-w\|_{H^1_0}\|z_t\|_{L^2} \\
& \leqslant L_0(\tau+s)(1+2\Gamma(s)e^{-\beta \varepsilon \tau} + 2r_0^2(s+\tau))\|z\|_{H^1_0}\|z_t\|_{L^2}.
\end{aligned}
\end{equation*}
Therefore, for $\tau\geqslant 0$, we obtain
\begin{align*}
\frac{d}{d\tau} & \|Z(\tau)\|^2_X \leqslant 2K_0\|z_t\|^2_{L^2}+2L_0(\tau+s)(1+2\Gamma(s)e^{-\beta \varepsilon \tau} + 2r_0^2(s+\tau))\|z\|_{H^1_0}\|z_t\|_{L^2}\\
& \stackrel{(1)}{\leqslant} 2K_0\|z_t\|^2_{L^2}+L_0(\tau+s)(1+2\Gamma(s)e^{-\beta \varepsilon \tau} + 2r_0^2(s+\tau))(\|z\|^2_{H^1_0}+\|z_t\|^2_{L^2}) \\
& \leqslant (2K_0+L_0(\tau+s)(1+2\Gamma(s)e^{-\beta \varepsilon \tau} + 2r_0^2(s+\tau))\|Z(\tau)\|^2_X,
\end{align*}
where in (1) we use again the usual Young's inequality. Since $L_0$, $\Gamma$ and $r_0$ are continuous, for each $\gamma>0$ we have
\[
c := \sup_{\tau\in [0,\gamma]} \Big(2K_0 + L_0(\tau+s)(1+2\Gamma(s)e^{-\beta \varepsilon \tau}+2r_0^2(s+\tau)\Big)<\infty,
\]
which implies that
\[
\frac{d}{d\tau}\|Z(\tau)\|_X^2 \leqslant c \|Z(\tau)\|_X^2.
\]
Applying Gronwall's inequality, we obtain
\begin{equation}\label{eq.Lip}
\|Z(\tau)\|_X \leqslant \|Z(0)\|_X e^{\frac{c}{2}\tau},
\end{equation}
for each $\tau\in [0,\gamma]$. Therefore, we have the following result:

\begin{theorem}\label{Lipschitz}
For each $\hat{D}\in \mathfrak{D}_{\mathcal{C}^\ast}$, $s\in \mathbb{R}$ and $\gamma>0$, there exists $L>0$ such that for all $V_1,V_2\in D_s$ and $0\leqslant \tau\leqslant \gamma$ we have
\[
\|S(\tau+s,s)V_1-S(\tau+s,s)V_2\|_X \leqslant L \|V_1-V_2\|_X.
\]
\end{theorem}
\begin{proof}
    It follows by noting that $Z(\tau)=S(\tau+s,s)V_1-S(\tau+s,s,V_2)$, $Z(0)=V_1-V_2$, and taking $L = e^{\frac{c}{2}\gamma}$ in \eqref{eq.Lip}.
\end{proof}

\bigskip \noindent \textbf{The generalized exponential pullback attractor: the proof of Theorem \ref{App:Att}} \medskip

Consider the family $\hat{C}$ given by \eqref{eq.def.C}. We already know that there exists $\Gamma \in \mathfrak{D}_{\mathcal{C}^\ast}$ such that for each $s\in \mathbb{R}$, $t\geqslant 0$ and $V_0\in C_s$, we have 
\[
\|V(t,s,V_0)\|_X^2 \leqslant \Gamma(s)e^{-\beta \varepsilon t} + r_0^2(t+s) \leqslant \Gamma(s) + r_0^2(t+s).
\]
If $T>0$ is fixed, setting
\[
\Xi_T(s):= \left\{\Gamma(s)+ \sup_{t\in [0,T]}r_0^2(t+s)\right\}^\frac12,
\]
if follows from Lemma \ref{supinC*} that $\Xi_T\in \mathcal{C}^\ast$. Thus, for each $s\in \mathbb{R}$, $V_0\in C_s$ and $t\in [0,T]$, we have
\[
\|V(t,s,V_0)\|_X \leqslant \Xi_T(s).
\]
If $V_1,V_2\in C_s$, $V(t,s,V_1)=(v(t),v_t(t))$ and $V(t,s,V_2)=(w(t),w_t(t))$, setting again $Z(t)=(z(t),z_t(t))=V(t,s,V_1)-V(t,s,V_2)$ we have $z=v-w$ and  we formally obtain from \eqref{eqondap} that
\begin{equation}\label{formallyztt}
z_{tt}-\Delta z+k_s(t)z_t+f_s(t,v)-f_s(t,w)=\int_{\Omega}K(x,y)z_t(t,y)dy.
\end{equation}

Our next goal is to use Theorem \ref{corMain2} to prove that the process $S$ associated with \eqref{ourproblem} possesses a generalized exponential pullback attractor. Note that for $s\in\mathbb{R}$ and $T>0$
\begin{equation}\label{dstmaiss2}
 \|S(T+s,s)V_1-S(T+s,s)V_2\|_X^2=\|Z(T)\|_X^2.
\end{equation}
Define $\mathcal{E}_s(\cdot)\colon \mathbb{R}^+ \to \mathbb{R}^+$ by
\[
\mathcal{E}_s(t)=\tfrac12\|Z(t)\|^2_X = \tfrac12(\|z\|^2_{H^1_0}+\|z_t\|^2_{L^2}),
\]
where $Z(t,Z_0)=(z(t),z_t(t))$ for $t\geqslant 0$. In order to obtain a suitable estimate for $\mathcal{E}_s$, we use the auxiliary function $\mathcal{F}_s$ defined by
\[
\mathcal{F}_s(t)=\mathcal{E}_s(t)+\tfrac{\varepsilon_0}{2}\langle z, z_t\rangle,
\]
where $\varepsilon_0:=\min\left\{\sqrt{\lambda_1}, k_0\right\}>0$. For $t\geqslant 0$ it is clear that
\begin{equation}\label{lema2000}
\tfrac{1}{2}\mathcal{E}_s(t)\leqslant \mathcal{F}_s(t)\leqslant \tfrac{3}{2}\mathcal{E}_s(t).
\end{equation}

Formally multiplying \eqref{formallyztt} by $z_t+\frac{\varepsilon_0}{2}z$ in $L^2(\Omega)$, we obtain
\begin{equation}\label{ecuacion367}
\begin{aligned}
    \frac{d}{dt}&\mathcal{F}_s(t)+\varepsilon_0\mathcal{F}_s(t)=-k_s(t)\|z_t\|^2_{L^2}+\varepsilon_0\|z_t\|^2_{L^2}+\frac{\varepsilon_0^2}{2}\langle z,z_t \rangle - \frac{\varepsilon_0}{2}k_s(t)\langle z, z_t\rangle\\
    &\quad-\langle f_s(t,v)-f_s(t,w), z_t\rangle-\frac{\varepsilon_0}{2}\langle f_s(t,v)-f_s(t,w), z \rangle\\
    &\quad+ \int_{\Omega\times \Omega}K(x,y)z_t(t,y)z_t(t,x)dydx+\frac{\varepsilon_0}{2}\int_{\Omega\times \Omega}K(x,y)z_t(t,y)z(t,x)dydx.  
\end{aligned}
\end{equation}

Now observe that for $t\in [0,T]$ we have
 \[ 
-k_s(t)\|z_t\|^2_{L^2} +\varepsilon_0\left\|z_t\right\|^2_{L^2} \leqslant 0,
\]
\[
\frac{\varepsilon_0^2}{2}\langle z,z_t \rangle \leqslant \frac{\varepsilon_0^2}{2}\left\|z_t\right\|_{L^2}\|z\|_{L^2}\leqslant \varepsilon_0^2\Xi_T(s)\|z\|_{L^2},
\]
\[
- \frac{\varepsilon_0}{2}k_s(t)\langle z, z_t\rangle \leqslant \varepsilon_0k_1\Xi_T(s)\|z\|_{L^2},
\]
\[ 
-\frac{\varepsilon_0}{2}\langle f_s(t,v)-f_s(t,w), z \rangle \leqslant \varepsilon_0L_0(t+s)\Xi_T(s)(1+2\Xi_T^2(s))\|z\|_{L^2} \leqslant \varepsilon_0 \Big(\sup_{t\in [0,T]}L_0(t+s)\Big)\Xi_T(s)(1+2\Xi_T^2(s))\|z\|_{L^2},
\]
\[
\int_{\Omega\times \Omega}K(x,y)z_t(t,y)z_t(t,x)dydx \leqslant 2\Xi_T(s)\left\|\int_{\Omega}K(x,y)z_t(t,y)dy\right\|_{L^2},
\]
\[
\frac{\varepsilon_0}{2}\int_{\Omega\times \Omega}K(x,y)z_t(t,y)z(t,x)dydx \leqslant \varepsilon_0K_0\Xi_T(s)\|z\|_{L^2}.
\]

Taking 
\begin{align*}
\varphi_T(s)= \Xi_T(s)\big[2+K_0\varepsilon_0+\epsilon_0^2+k_1\varepsilon_0+\varepsilon_0\sup_{t\in [0,T]}L_0(t+s)(1+2\Xi_T^2(s))\big],
\end{align*}
we can see that $\varphi_T\in \mathcal{C}^\ast$ and inserting the previous estimates into \eqref{ecuacion367}, for $t\in [0,T]$ we conclude that
\begin{align*}
\frac{d}{dt}\mathcal{F}_s(t) + \varepsilon_0\mathcal{F}_s(t)\leqslant \varphi_T(s)\|z\|_{L^2}-\langle f_s(t,v)-f_s(t,w), z_t \rangle + \varphi_T(s)\left\|\int_{\Omega}K(x,y)z_t(t,y)dy\right\|_{L^2}.
\end{align*}

\begin{lemma}\label{lema3000}
    Let $T>0$ fixed. Then,
    \begin{align*}
        \mathcal{E}_s(T)&\leqslant 3e^{-\varepsilon_0T}\mathcal{E}_s(0)+2\varphi_T(s)T\sup_{t\in [0,T]}\left\|z(t)\right\|_{L^2}+2\varphi_T(s)\int_0^T\left\|\int_{\Omega}K(x,y)z_t(t,y)dy\right\|_{L^2}dt\\
        &\quad+2\left|\int_0^T e^{\varepsilon_0(t-T)}\langle f_s(t,v)-f_s(t,w), z_t(t) \rangle dt\right|.
    \end{align*}
\end{lemma}

\begin{proof}
    For every $t\geqslant 0$ we have
    \begin{align*}
        \frac{d}{dt}(e^{\varepsilon_0 t}\mathcal{F}_s(t))&=e^{\varepsilon_0 t}\left(\frac{d}{dt}\mathcal{F}_s(t)+\varepsilon_0\mathcal{F}_s(t)\right)\\
        &\leqslant e^{\varepsilon_0 t}\varphi_T(s)\left\|z\right\|_{L^2}-e^{\varepsilon_0 t}\langle f_s(t,v)-f_s(t,w), z_t \rangle + e^{\varepsilon_0 t}\varphi_T(s)\left\|\int_{\Omega}K(x,y)z_t(t,y)dy\right\|_{L^2}.
    \end{align*}
Now, an integration from $0$ to $T$ provides
\begin{align*}
    e^{\varepsilon_0 T}\mathcal{F}_s(T)-\mathcal{F}_s(0) &\leqslant e^{\varepsilon_0 T}\varphi_T(s)\int_0^T\left\|z(t)\right\|_{L^2}dt+e^{\varepsilon_0 T}\varphi_T(s)\int_0^T\left\|\int_{\Omega}K(x,y)z_t(t,y)dy\right\|_{L^2}dt\\
    &\quad+\left|\int_0^T e^{\varepsilon_0 t}\langle f_s(t,v)-f_s(t,w),z_t(t) \rangle dt\right|,
\end{align*}
which implies
\begin{align*}
    \mathcal{F}_s(T)&\leqslant e^{-\varepsilon_0 T}\mathcal{F}_s(0) +\varphi_T(s)\int_0^T\left\|z(t)\right\|_{L^2}dt+\varphi_T(s)\int_0^T\left\|\int_{\Omega}K(x,y)z_t(t,y)dy\right\|_{L^2}dt\\
    &\quad+\left|\int_0^T e^{\varepsilon_0 (t-T)}\langle f_s(t,v)-f_s(t,w),z_t(t) \rangle dt\right|,
\end{align*}
and the proof is done if we apply \eqref{lema2000} to the previous inequality and consider that $\int_0^T\|z(t)\|_{L^2}dt\leqslant T\sup_{t\in [0,T]}\|z(t)\|_{L^2}$.
\end{proof}

From \eqref{dstmaiss2} and Lemma \ref{lema3000}, we conclude that for $t\in \mathbb{R}$, $n\in \mathbb{N}$ and $T>0$, writing $S_n = S(t-(n-1)T,t-nT)$, we have,
\begin{align*}
d&(S_nV_1, S_nV_2)^2 \leqslant 3e^{-\varepsilon_0T}d(V_1,V_2)^2+4T\varphi_T(t-nT)\sup_{t\in [0,T]}\|z(t)\|_{L^2}\\ 
&\quad+4\varphi_T(t-nT)\int_0^T\left\|\int_{\Omega}K(x,y)z_t(t,y)dy\right\|_{L^2}dt+4\left|\int_0^T e^{\varepsilon_0(t-T)}\langle f_{t-nT}(t,v)-f_{t-nT}(t,w), z_t(t) \rangle dt\right|.
\end{align*}
Note that $\mu:= 3e^{-\varepsilon_0 T}\in (0,1)$ if we take $T>\varepsilon_0^{-1}\ln 3$. For instance, taking $T=1+\varepsilon_0^{-1}\ln 3$, we have
\[
\mu=3e^{-\varepsilon_0(\varepsilon_0^{-1}ln 3 + 1)}=3e^{-ln3}e^{-\varepsilon_0}=e^{-\min\left\{\sqrt{\lambda_1},k_0\right\}}.
\]

Consider the functions $g_n\colon \mathbb{R}^+\times \mathbb{R}^+\to \mathbb{R}$ given by $g_n(\alpha, \beta)=4\varphi_T(t-nT)\alpha+4T\varphi_T(t-nT)\beta$. It is clear that each $g_n$ is non-decreasing with respect to each variable, $g_n(0,0)=0$ and $g_n$ is continuous at $(0,0)$. Additionally, define the maps $\rho_1,\rho_2, \psi_n\colon X\times X\to \mathbb{R}^+$ by
\begin{align*}    &\rho_1(V_1,V_2)=\int_0^T\left\|\int_{\Omega}K(x,y)z_t(\tau,y)dy\right\|_{L^2}d\tau,\\
&\rho_2(V_1,V_2)=\sup_{\tau\in [0,T]}\|z(\tau)\|_{L^2},\\
&\psi_n(V_1,V_2)=4\left|\int_0^T e^{\varepsilon_0(\tau-T)}\langle f_{t-nT}(\tau,v)-f_{t-nT}(\tau,w), z_t(\tau) \rangle dt\right|,
\end{align*}
where $(v(\tau),v_t(\tau))=V(\tau,t-nT,V_1)$, $(w(t),w_t(t))=V(\tau,t-nT,V_2)$ and $(z(\tau),z_t(\tau))=Z(\tau)=V(\tau,t-nT,V_1)-V(\tau,t-nT,V_2)$.

We already know that the evolution process $S$ associated with \eqref{ourproblem} satisfies the Lipschitz condition required in Theorem \ref{corMain2} (see Theorem \ref{Lipschitz}). We need to verify that $\psi_n\in \operatorname{contr}(C_{t-nT})$, and also that $\rho_1,\rho_2$ are precompact in $C_{t-nT}$ for each $t\in \mathbb{R}$ and $n\in \mathbb{N}$. 

We begin by proving that $\psi_n\in \operatorname{contr}(C_{t-nT})$. Let $\left\{V_k\right\}_{k\in\mathbb{N}}\subset C_{t-nT}$ and denote $V(t,t-nT,V_k)=(v^{(k)}(t),v_t^{(k)}(t))$. As in \cite{Pecorari}, we have the following result:

\begin{lemma}\label{convhs}
For a fixed $\gamma\in (0,1)$, up to a subsequence, we have
\begin{equation*}
\left\{\begin{aligned}
    &(v^{(k)},v_t^{(k)})\overset{\ast}{\rightharpoonup} (v,v_t) \hbox{ in }L^{\infty}(0,T; X),\\
    & v^{(k)}\rightarrow v \hbox{ in } C([0,T]; H^\gamma(\Omega)).
    \end{aligned}\right.
\end{equation*}

Also for each $\ell\in [0,T]$ and $\psi \in L^1(0, T; L^2(\Omega))$ it holds that
   $$\int_\ell^T \langle f_{t-nT}(\tau, w^{(k)}(\tau))-f_{t-nT}(\tau, w(\tau)), \psi(\tau)\rangle d\tau \rightarrow 0$$
as $k\rightarrow \infty$, and
\[
f_{t-nT}(\tau, v^{(k)}(\tau))\stackrel{\ast}{\rightharpoonup} f_{t-nT}(\tau, v(\tau)) \hbox{ in } L^{\infty}(t, T; L^2(\Omega)) \hbox{ as } k\to \infty.
\]
\end{lemma}

Now, we have:

\begin{lemma}\label{proppart1e}
For $k\in \mathbb{N}$ we have
\[
\int_0^Te^{\varepsilon_0(\tau-T)}\langle f_{t-nT}(\tau, v^{(k)}(\tau)),v_t^{(m)}(\tau)\rangle d\tau \stackrel{m\to \infty}{\longrightarrow} \int_0^T e^{\varepsilon_0(\tau-T)} \langle  f_{t-nT}(\tau, v^{(k)}(\tau)),v_t(\tau)\rangle d\tau,
\]
and
\[
\int_0^T e^{\varepsilon_0(\tau-T)}\langle f_{t-nT}(\tau, v^{(m)}(\tau)), v_t^{(k)}(\tau) \rangle d\tau \stackrel{m\to \infty}{\longrightarrow} \int_0^T e^{\varepsilon_0(\tau-T)}\langle f_{t-nT}(\tau, v(\tau)),v_t^{(k)}(\tau)\rangle d\tau.
\]
\end{lemma}
\begin{proof}
For the first item, note that defining $\xi\colon [0,T]\to L^2(\Omega)$ by $\xi(\tau)=e^{\varepsilon_0(\tau-T)}f_{t-nT}(\tau, v^{(k)}(\tau))$ for each $\tau \in [0,T]$, we have $\xi \in L^1(0, T; L^2(\Omega))$, since
\begin{align*}
\int_0^T \|\xi(\tau)\|_{L^2} d\tau &= \int_0^T \left\|e^{\varepsilon_0(\tau-T)}f_{t-nT}(\tau,v^{(k)}(\tau))\right\|_{L^2} d\tau \\
&= \int_0^T \underbrace{\left|e^{\varepsilon_0(\tau-T)}\right|}_{\leqslant 1}\left\|f_{t-nT}(\tau,v^{(k)}(\tau))\right\|_{L^2} d\tau \leqslant T\delta(t-nT)<\infty,
\end{align*}
for some $\delta\in \mathcal{C}^\ast$. Since, by Lemma \ref{convhs}, $v_t^{(m)} \stackrel{\ast}{\rightharpoonup} v_t$ in $L^{\infty}(0, T; L^2(\Omega))$ as $m\to \infty$ we have 
\[
\int_0^T\langle \xi(\tau), v^{(m)}_t(\tau)-v_t(\tau) \rangle d\tau \stackrel{m\to \infty}{\longrightarrow} 0.
\]

For the second item, define $\xi\colon [0,T]\to L^2(\Omega)$ by $\xi(\tau) = e^{\varepsilon_0(\tau-T)}v_t^{(k)}(\tau)$ for each $\tau\in [0,T]$. Clearly $\xi\in L^1(0,T;L^2(\Omega))$, since for some $\delta\in \mathcal{C}^\ast$ we have
\[
\int_0^T \|\xi(\tau)\|_{L^2}d\tau = \int_0^t \|e^{\varepsilon_0(\tau-T)}v_t^{(k)}(\tau)\|_{L^2}d\tau = \int_0^t \underbrace{|e^{\varepsilon_0(\tau-T)}|}_{\leqslant 1}\|v_t^{(k)}(\tau)\|_{L^2}d\tau \leqslant T\delta(t-nT)<\infty,
\]
and hence, since $f_{t-nT}(\tau,v^{(m)}(\tau))\stackrel{\ast}{\rightharpoonup} f_{t-nT}(\tau,v(\tau))$ in $L^\infty(0,T;L^2(\Omega))$ as $m\to \infty$ by Lemma \ref{convhs}, we have
\[
\int_0^T \langle f_{t-nT}(\tau,v^{(m)}(\tau))-f_{t-nT}(\tau,v(\tau)),\xi(\tau)\rangle d\tau \stackrel{m\to \infty}{\longrightarrow} 0,
\]
and the result is proven.
\end{proof}

Analogously, we can show the following:

\begin{lemma}\label{proppart2e}
\[
\int_0^T e^{\varepsilon_0(\tau-T)} \langle f_{t-nT}(\tau, v^{(k)}(\tau)),v_t(\tau)\rangle d\tau \stackrel{k\to \infty}{\longrightarrow} \int_0^T e^{\varepsilon_0(\tau-T)} \langle f_{t-nT}(\tau, v(\tau)),v_t(\tau)\rangle d\tau,
\]
and
\[
\int_0^T e^{\varepsilon_0(\tau-T)}\langle f_s(\tau, v(\tau)),v^{(k)}_t(\tau)\rangle d\tau \stackrel{k\to \infty}{\longrightarrow} \int_0^T e^{\varepsilon_0(\tau-T)}\langle f_s(\tau, v(\tau)),v_t(\tau)\rangle d\tau.
\]
\end{lemma}

Using these two lemmas, we obtain:

\begin{lemma}\label{propt2023xexp}
\begin{align*}
&\lim_{k\to \infty}\lim_{m\to \infty}\int_0^T e^{\varepsilon_0(\tau-T)}\langle f_{t-nT}(\tau, v^{(k)}(\tau)),v_t^{(m)}(\tau)\rangle d\tau\\
&\quad=\lim_{k\to \infty}\lim_{m\to \infty}\int_0^T e^{\varepsilon_0(\tau-T)}\langle f_{t-nT}(\tau, v^{(m)}(\tau)),v_t^{(k)}(\tau)\rangle d\tau\\
&\quad=\int_{\Omega}F_{t-nT}(T,v(T))dx-e^{-\varepsilon_0 T}\int_{\Omega}F_{t-nT}(0,v(0))dx\\
&\qquad-\varepsilon_0\int_{\Omega}\int_0^T e^{\varepsilon_0(\tau-T)}F_{t-nT}(\tau,v(\tau))d\tau dx-\int_{\Omega}\int_0^T e^{\varepsilon_0(\tau-T)}\frac{\partial}{\partial \tau}F_{t-nT}(\tau, v(\tau)) d\tau dx.
\end{align*}
\end{lemma}
\begin{proof}
We have
\begin{align*}
    &\lim_{k\to \infty}\lim_{m\to \infty}\int_0^T e^{\varepsilon_0(\tau-T)}\langle f_{t-nT}(\tau, v^{(k)}(\tau)),v_t^{(m)}(\tau)\rangle d\tau\\
&\quad=\lim_{k\to \infty}\lim_{m\to \infty}\int_0^T e^{\varepsilon_0(\tau-T)}\langle f_{t-nT}(\tau, v^{(m)}(\tau)),v_t^{(k)}(\tau)\rangle d\tau =\int_0^T e^{\varepsilon_0(\tau-T)}\langle f_{t-nT}(\tau,v(\tau)), v_t(\tau)\rangle d\tau.
\end{align*}

Since for $\tau\in [0,T]$ we have
\[
\int_{\Omega}|e^{\varepsilon_0(\tau-T)}f_{t-nT}(\tau, v(\tau))v_t(\tau)| dx \leqslant \|f_{t-nT}(\tau, v(\tau))\|_{L^2}\|v_t(\tau)\|_{L^2}\leqslant \delta(t-nT)<\infty,
\]
for some $\delta\in \mathcal{C}^\ast$, we can apply Fubini's Theorem to obtain 
\begin{align*}
\int_0^T \int_{\Omega}&e^{\varepsilon_0(\tau-T)}f_{t-nT}(\tau, v(\tau))v_t(\tau)dx d\tau = \int_{\Omega}\int_0^T e^{\varepsilon_0(\tau-T)}f_{t-nT}(\tau, v(\tau))v_t(\tau) d\tau dx.
\end{align*}
Note that 
\begin{align*}
    \frac{d}{d\tau}\left[\frac{1}{\varepsilon_0}e^{\varepsilon_0(\tau-T)}F_{t-nT}(\tau,v(\tau))\right]&=e^{\varepsilon_0(\tau-T)}F_{t-nT}(\tau,v(\tau))+\frac{1}{\varepsilon_0}e^{\varepsilon_0(\tau-T)}\frac{\partial}{\partial \tau}F_{t-nT}(\tau,v(\tau))\\
    &\quad+\frac{1}{\varepsilon_0}e^{\varepsilon_0(\tau-T)}f_{t-nT}(\tau,v(\tau))v_t(\tau),
\end{align*}
which, integrating from $0$ to $T$, implies
\begin{align*}
    \int_0^T & e^{\varepsilon_0(\tau-T)}f_{t-nT}(\tau,v(\tau))v_t(\tau) d\tau= \int_0^T  \frac{d}{d\tau}[e^{\varepsilon_0(\tau-T)}F_{t-nT}(\tau,v(\tau))] d\tau\\
    &\quad-\varepsilon_0\int_0^T e^{\varepsilon_0(\tau-T)}F_{t-nT}(\tau,v(\tau))d\tau-\int_0^Te^{\varepsilon_0(\tau-T)}\frac{\partial}{\partial \tau}F_{t-nT}(\tau,v(\tau))d\tau.
\end{align*}
Now the conclusion follows immediately.
\end{proof}

Thus, we obtain:

\begin{corollary}
We have
\[
\lim_{k\to \infty}\lim_{m\to \infty}\int_0^T e^{\varepsilon_0(\tau-T)}\langle f_{t-nT}(\tau,v^{(k)}(\tau))-f_s(\tau,v^{(m)}(\tau)),v_t^{(n)}(\tau)-v_t^{(m)}(\tau)\rangle d\tau =0.
\]
\end{corollary}

We can finally prove Theorem \ref{App:Att}.

\begin{proof}[Proof of Theorem \ref{App:Att}]
With these previous results we conclude the contractiveness of $\psi_n$ in $C_{t-nT}$ for each $n\in \mathbb{N}$. Applying similar changes to the proof of the results presented in \cite{Pecorari}, we can prove that $\rho_1,\rho_2$ are precompact on $C_{t-nT}$.

Therefore, Theorem \ref{corMain2} ensures the existence of a generalized exponential $\mathfrak{D}_{\mathcal{C}^\ast}$-pullback attractor $\hat{M}$ for the evolution process $S$ associated with \eqref{ourproblem}, with $\hat{M}\in \mathfrak{D}_{\mathcal{C}^\ast}$. Furthermore, from Proposition \ref{theo:GenimpliesPA}, it follows that $S$ has a $\mathfrak{D}_{\mathcal{C}^\ast}$-pullback attractor with $\hat{A}\subset \hat{M}$. 
\end{proof}

\medskip \textbf{Data Availability.} Data sharing is not applicable to this article as no new data were created or analyzed in
this study.

\section*{Declarations}

\textbf{Competing Interests.} The authors have not disclosed any competing interests.

\begin{appendices}

\section{Auxiliary results}\label{appendix}

In this brief appendix, we present two results regarding properties of $\mathcal{C}^\ast$.

\begin{lemma}\label{integralemC*}
    If $\eta>0$ and $\delta\in \mathcal{C}^\ast$, then the map
    $
    \mathbb{R}\ni t \mapsto \mu(t):=\int_0^\infty \delta(t-u)e^{-\eta u}du
    $
    is also in $\mathcal{C}^\ast$.
\end{lemma}
\begin{proof}
    Fix $\epsilon>0$. Since $\delta\in \mathcal{C}^\ast$, given $\alpha>0$ and $t\in \mathbb{R}$ there exists $\tau_0>0$ such that for all $\tau\geqslant \tau_0$ and $\xi\leqslant t$ we have
    $
    \delta(\xi-\tau)e^{-\alpha \tau} < \eta \epsilon.
    $
    Thus for $s\leqslant t$
    \[
    \mu(s-\tau)e^{-\alpha \tau}=e^{-\alpha \tau}\int_0^\infty \delta(s-\tau-u)e^{-\eta u}du = \int_0^\infty \underbrace{\delta(s-u-\tau)e^{-\alpha \tau}}_{<\eta\epsilon}e^{-\eta u}du < \eta \epsilon \int_0^\infty e^{-\eta u}du = \epsilon,
    \]
    which implies that
    $
    \limsup_{\tau\to \infty}\sup_{s\leqslant t}\mu(s-\tau)e^{-\alpha \tau} \leqslant \epsilon.
    $
    Since $\epsilon>0$ is arbitrary, we conclude that $\mu\in \mathcal{C}^\ast$.
\end{proof}

\begin{lemma}\label{supinC*}
If $\delta\in \mathcal{C}^\ast$ and $T>0$ then the map $\mu\colon \mathbb{R}\to (0,\infty)$, defined by
$
\mu(s) = \sup_{\ell \in [0,T]} \delta(\ell+s) \hbox{ for } s\in \mathbb{R}, 
$
is in $\mathcal{C}^\ast$.
\end{lemma}
\begin{proof}
    Since $\delta\in \mathcal{C}^\ast$, given $t\in \mathbb{R}$, $\alpha>0$ and $\epsilon>0$, there exists $\tau_0\geqslant 0$ such that $\delta(s-\tau)e^{-\alpha \tau}<\epsilon$ for all $\tau\geqslant \tau_0$ and $s\leqslant t+T$. Hence if $\ell\in [0,T]$ and $s\leqslant t$ we have $\ell+s\leqslant t+T$ and
    $
    \delta(\ell+s-\tau)e^{-\alpha \tau} < \epsilon,
    $
    which implies that 
    $
    \mu(s-\tau)e^{-\alpha \tau}\leqslant \epsilon \hbox{ for all } \tau \geqslant \tau_0 \hbox{ and } s\leqslant t.
    $
    Thus $\mu\in \mathcal{C}^\ast$ as claimed.
\end{proof}

\end{appendices}

\bibliographystyle{abbrv}

\end{document}